\documentclass[11pt]{amsart}
\usepackage[all]{xy}
\usepackage{enumitem}
\usepackage{mathtools}
\usepackage{mathrsfs}
\usepackage{amssymb}
\usepackage{amsmath,amsfonts,amsthm}
\usepackage{graphicx}
\usepackage{upgreek}
\usepackage[toc,page]{appendix}
\usepackage{bbm}
\usepackage{esint}
\usepackage{scalerel}
\usepackage{stackengine,wasysym}
\usepackage{verbatim}
\usepackage{hyperref}
\usepackage[dvipsnames]{xcolor}
\hypersetup{colorlinks=true, linktocpage=true,citecolor=Emerald, linkcolor=RedViolet, urlcolor=RedViolet}

\usepackage{geometry}
\numberwithin{equation}{section}
\newtheorem{teo}{Theorem}[section]
\newtheorem{lema}[teo]{Lemma}

\newtheorem{prop}[teo]{Proposition}
\newtheorem{coro}[teo]{Corollary}

\newtheorem{obs}[teo]{Remark}

\title{Finite index constant mean curvature hypersurfaces in low dimensions}

\author{Ivan Miranda}
\address{IMPA -- Instituto de Matemática Pura e Aplicada, Rio de Janeiro, RJ, Brasil, 22460-320.}
\email{ivan.miranda@impa.br}

\begin{document}
	
	\begin{abstract}
    		    We prove that every complete finite index immersed CMC hypersurface is either minimal or compact, provided that the ambient six-dimensional manifold is a Riemannian product of a closed manifold with non-negative sectional curvature and a Euclidean factor. This answers affirmatively a question posed by do Carmo, for this class of ambient spaces, and extends known lower dimensional results. As a consequence, we complete the classification of two-sided, complete weakly stable CMC hypersurfaces immersed in the space forms of positive curvature in dimension six. More generally, we study the class of Riemannian manifolds with bounded curvature and obtain several partial results. In particular, we show that a complete, finite index CMC hypersurface immersed in the hyperbolic six-space with mean curvature vector of length greater than seven is necessarily compact. 
	\end{abstract}

	\maketitle
	
	\setcounter{tocdepth}{1}
	\tableofcontents
	
	\section{Introduction}

The isoperimetric problem concerns the search for hypersurfaces that have the least
    area among all of those that enclose the same volume. A minimizer for this problem is called an isoperimetric hypersurface, and the study of these objects is a classical topic in Geometry (\cite{Isoperimetric}, \cite{ros2005isoperimetric}). Connecting this problem with the Calculus of Variations, one deduces infinitesimal information for a hypersurface that minimizes area up to second order among competitors that enclose the same volume. The first variation shows that a minimizer has constant mean curvature (CMC), and the second variation can be encoded as the non-negativity of a quadratic form restricted to a functional space.
    
    In general, a non-compact hypersurface does not bound a finite volume, but one can still make sense of variations that preserve the enclosed volume if the hypersurface is two-sided, by means of defining a notion of algebraic volume between the hypersurface and its compactly supported perturbations. A two-sided CMC hypersurface $M$ that minimizes area up to second order among variations with compact support that preserve the enclosed volume is called \textit{weakly stable}, and this condition is satisfied if and only if the hypersurface satisfies the weak stability inequality
        \begin{equation}
                \int_M | \nabla \phi|^2 - (Ric(\nu)+|A|^2)\phi^2 \ge 0 \label{desig de estabilidade}
        \end{equation}

            \noindent for every $\phi \in C^\infty_0(M)$ with zero mean value, \textit{i.e.} $\int_M \phi =0$.

    Related to the notion of stability is the Morse index of a constant mean curvature hypersurface. For every precompact domain $\Omega$ of $M$, the quadratic form \eqref{desig de estabilidade}, restricted to functions $\phi \in C^\infty_0(\Omega)$ with zero mean value, has finite index which we call the \textit{weak index} of $\Omega$. The weak index of $M$, denoted by Ind$_0(M)$, is the supremum over precompact domains $\Omega \subset M$ of the weak index of $\Omega$. We say that $M$ has finite index when Ind$_0(M)$ is finite.

    Our main result concerns finite index CMC hypersurfaces immersed in ambient spaces with non-negative sectional curvature.

	{
		\renewcommand{\theteo}{A}
		\begin{teo} \label{quarto teorema principal k > 0}
        Let $N$ be a closed Riemannian manifold with non-negative sectional curvature, and dimension $k \in \{0,\dots,6\}$. Consider the Riemannian product $X := N^k \times \mathbb{R}^{6-k}$ between $N$ and a Euclidean factor. Every complete CMC hypersurface with finite index immersed in $X$ is either minimal or compact.
		\end{teo}                                  
	}

        In 1989 \cite{doCarmoLectureNotes}, M. P. do Carmo asked whether every complete, non-compact, weakly stable CMC hypersurface immersed in Euclidean space is necessarily minimal. Theorem \ref{quarto teorema principal k > 0} shows that do Carmo's question has a positive answer in a more general class of six-dimensional ambient spaces, even if one relaxes the hypothesis of stability to the one of finite index. This extends lower dimensional results (\textit{cf.} the survey of B. Nelli \cite{SurveyNelli}). 
        
        Our proof of Theorem \ref{quarto teorema principal k > 0} is based on the recent solution of the stable Bernstein problem in $\mathbb{R}^6$ by L. Mazet \cite{Mazet} and provides an alternative route to the positive answer to do Carmo's question in the six-dimensional Euclidean space by J. Chen, H. Hong, and H. Li \cite{CHL}.

        We apply Theorem \ref{quarto teorema principal k > 0} to complete the classification of weakly stable CMC hypersurfaces in the six-dimensional space forms of positive curvature.

	{
		\renewcommand{\theteo}{B}
		\begin{coro} \label{corolario de classificacao na esfera}
			    Every complete, two-sided, weakly stable CMC hypersurface immersed in the round sphere $\mathbb{S}^6$ is a geodesic sphere. 
		\end{coro}                               
	}
        
    	{
		\renewcommand{\theteo}{C}
		\begin{coro} \label{corolario de classificacao no projetivo}
    			    Every complete, two-sided, weakly stable CMC hypersurface immersed in the round real projective space $\mathbb{R}\mathbb{P}^6$ is either a geodesic sphere, or a quotient of a Clifford hypersurface, or the two-fold covering of a projective subspace $\mathbb{R}\mathbb{P}^5$.
		\end{coro}                               
	}

 Compact weakly stable CMC hypersurfaces immersed in $\mathbb{S}^6$ and $\mathbb{R}\mathbb{P}^6$ were classified by J. L. Barbosa, M. P. do Carmo, and J. Eschenburg \cite{Barbosa-Manfredo-Eschenburg}, and C. Viana \cite{CelsoRPn}. We use Theorem \ref{quarto teorema principal k > 0} to prove compactness in the non-minimal case, and Theorem 1.2 by G. Catino, P. Mastrolia and A. Roncoroni \cite{CatinoStable} in the minimal case (see also Theorem \ref{corolario da extensao do teorema do Catino}).  \\

The ambient Riemannian manifolds $X$ considered in Theorem \ref{quarto teorema principal k > 0} are special in that their isometry group acts cocompactly on $X$, \textit{i.e.} there exists a compact subset $K \subset X$ that intersects all orbits. Our next result concerns the compactness of finite index CMC hypersurfaces immersed in ambient spaces with this geometric property. 

	{
		\renewcommand{\theteo}{D}  
		\begin{teo}  \label{quarto teorema principal k>-1}
            \,Let $X$ be a six-dimensional Riemannian manifold with sectional curvature bounded from below by $-1$. Suppose that the isometry group of $X$ acts cocompactly on $X$. Every complete finite index CMC hypersurface immersed in $X$ with mean curvature $|H| > 7$ is compact.
		\end{teo}                               
	}

    As a corollary, we contribute to the classification of weakly stable CMC hypersurfaces in the six-dimensional hyperbolic space. 
    
     	{
		\renewcommand{\theteo}{E}  
		\begin{coro} \label{corolario de classificacao no espaco hiperbolico}
           Every complete weakly stable CMC hypersurface immersed in the six-dimensional hyperbolic space $\mathbb{H}^6$ with mean curvature $|H| > 7$ is a geodesic sphere.
		\end{coro}                               
	}

        The proof of Corollary \ref{corolario de classificacao no espaco hiperbolico} shows the non-existence of non-compact hypersurfaces satisfying the hypotheses of the theorem, and extends lower dimensional results (see Appendix \ref{section::Lower dimensions}). The case of compact hypersurfaces was treated by J. L. Barbosa, M. P. do Carmo, and J. Eschenburg \cite{Barbosa-Manfredo-Eschenburg} in every dimension. \\
        
        We say that a CMC hypersurface $M$ is \textit{strongly stable} when \eqref{desig de estabilidade} is satisfied for every $\phi \in C^\infty_0(M)$. In comparison with the weak stability condition, we have dropped the zero mean value condition on the test functions. (For motivations, see for instance the introduction of the work of J. L. Barbosa and M. P. do Carmo \cite{Barbosa-Manfredo} and the work of A. Ros and H. Rosenberg \cite{Ros-Rosenberg}.) Analogous to the definition of \textit{weak index} is the definition of \textit{strong index} for a CMC hypersurface, denoted by Ind$(M)$. Clearly, Ind$_0(M) \le$ Ind$(M) \le $ Ind$_0(M)+1$. 
        
        The horospheres of the hyperbolic space $\mathbb{H}^{n+1}$ are complete strongly stable CMC hypersurfaces with mean curvature $|H| = n$. Thus Corollary \ref{corolario de classificacao no espaco hiperbolico} gives a partial answer to the following question posed by O. Chodosh in his recent survey for the ICM (\textit{cf.} \cite{Chodosh-Survey}, p. 19): If a complete strongly stable CMC hypersurface immersed in the hyperbolic space $\mathbb{H}^{n+1}$ has mean curvature $|H| \ge n$, must it be a horosphere?  This question has a positive answer when $n+1 = 3$, by work of A. da Silveira \cite{AlexandreDaSilveira}.

    It is well known that the classification of strongly stable minimal hypersurfaces in Euclidean space is directly related to curvature estimates for compact strongly stable CMC hypersurfaces $M$, with nonempty boundary, immersed in an ambient space $X$ with well behaved geometry. These estimates concern upper bounds for the second fundamental form of the immersion at an interior point by its distance to the boundary of $M$ and a constant $C$ that depends on the geometry of the ambient space $X$ and the mean curvature of $M$, $C = C(X,H)$. In the seminal work of H. Rosenberg, R. Souam, and E. Toubiana \cite{Rosenberg-Souam-Toubiana}, the authors use the classification of strongly stable CMC hypersurfaces in Euclidean three-space and results from the work of E. Hebey and M. Herzlich \cite{Hebey-Herzlich} to prove that the constant $C$ may be chosen to depend only on bounds on the sectional curvature of $X$. We have extended this result to higher dimensions.

    	{
		\renewcommand{\theteo}{F}
            \begin{teo} \label{estimativa de curvatura mais geral possivel na intro}
                Let $X$ be a complete Riemannian manifold of dimension $3 \le n+1 \le 6$. Suppose that $X$ has bounded sectional curvature $|sec_X| \le K$, for some positive real number $K$. Then there exists a constant $C=C(K) > 0$, that depends only on $K$, with the following property. Any compact, connected, two-sided and strongly stable CMC hypersurface $M$ with nonempty boundary immersed in $X$ satisfies that
                \begin{equation*}
                    |A(p)|\min \{ d_M(p,\partial M), 1 \} \le C, \quad \quad \forall\, p \in M.
                \end{equation*}
    \end{teo}                          
	}

    Theorem \ref{estimativa de curvatura mais geral possivel na intro} was proved for minimal hypersurfaces immersed in four-dimensional manifolds in the work of O. Chodosh, C. Li, and D. Stryker \cite{Chodosh-Li-Stryker-4mfds}, which settled a conjecture of R. Schoen.

\subsection{Main ideas} \label{subsection::Some ideas behind the proofs of the main theorems}

         The proof of Theorem \ref{quarto teorema principal k > 0} is by contradiction. Suppose that $M'$ is a complete, non-compact, non-minimal finite index CMC hypersurface immersed in $X$. As we go towards an end of $M'$ we construct a sequence of strongly stable intrinsic metric balls of $M'$, $B(p_k,r_k)$, with increasing radii $r_k \to +\infty$. We denote by $\psi$ the immersion of $M'$ into $X$ and study the sequence of immersions $\psi_k:B(p_k,r_k) \to X$. The symmetries of the ambient space are used to translate, without renaming, each $\psi_k$ so that they intersect a fixed compact subset of $X$. The idea is to pass to a subsequence and construct a limit immersion $\psi_\infty:M \to X$ (\textit{cf.} the \textit{Reduction Lemma}, Theorem \ref{argumento de redução em homogeneas}). To this end, we use the recent solution to the stable Bernstein problem in $\mathbb{R}^6$ by L. Mazet \cite{Mazet} to derive curvature estimates for the immersions $\psi_k$. The produced immersion is a non-compact, complete CMC immersion with the same mean curvature as $M'$, but it is strongly stable and has bounded second fundamental form.
                     
                 Next, we study the geometry and topology of $M$. For instance, it is known that it must have infinite volume. Moreover, the reduction to the strongly stable case allows for topological simplifications: we may pass to the universal cover and assume it is simply connected, and we can guarantee it has only one end.
                 
                 The theory of $\mu$-bubbles can be used to further restrict the geometry of $M$ in the following way. After some computations with the stability inequality of $M$, we prove that $M$ has uniformly positive $\alpha$-bi-Ricci curvature in a weak sense, a notion introduced in the work of L. Mazet \cite{Mazet}. This curvature condition, together with the simplified topology of $M$, allows us to construct an exhaustion $\Omega_j$ of $M$ by smooth precompact domains so that each $\Omega_j$ has exactly one boundary component $\Sigma_j$, and each $\Sigma_j$ has area bounded from above by a constant that depends only on $M$, and not on $j$ (\textit{cf.} Theorem \ref{Aplicacao à geometria Riemanniana}). 
                 
                 This geometric property of $M$ and the fact that it has infinite volume clearly guarantee that its Cheeger constant is zero. But this contradicts P. Buser's inequality (\cite{Buser}, Theorem 7.1), because the strong stability inequality of $M$ readily implies that $\lambda_1(M)>0$. 

                     The proof of Theorem \ref{quarto teorema principal k>-1} follows a similar strategy. The main difference is that, under the negative lower bound on the ambient sectional curvature, both the estimates on the number of ends of a non-compact strongly stable CMC hypersurface (\textit{cf}. Theorem \ref{unicidade de fim quando curvature é limitada inferiormente em X6}) and the estimates on the area of $\mu$-bubbles under positive $\alpha$-bi-Ricci weak lower bounds (\textit{cf}. Proposition \ref{controle do bi-Ricci com peso da CMC estavel em ambiente de curvatura negativa}) require the mean curvature to satisfy an inequality of the form $|H|>5+\varepsilon$, for some explicit $\varepsilon>0$. While the choice $\varepsilon =2$ works,  it is unclear if one could achieve the desired lower bound $|H| >5$ using the current methods (see more about that in Remark \ref{remark sobre os fins das cmc no hiperbolico} and Remark \ref{OBS sobre curvature critica otima em H6}).

\subsection{Related works}

                    Our work was strongly influenced by several previous studies, which we now highlight. H. Rosenberg, R. Souam, and E. Toubiana \cite{Rosenberg-Souam-Toubiana} investigated curvature estimates for CMC surfaces. O. Chodosh, C. Li, and D. Stryker \cite{Chodosh-Li-Stryker-4mfds} studied stable minimal hypersurfaces in positively curved four-manifolds. L. Mazet \cite{Mazet} classified stable minimal hypersurfaces in Euclidean six-space. Finally, J. Chen, H. Hong, and H. Li \cite{CHL} provided a positive answer to do Carmo's question in the six-dimensional Euclidean space.
                    
                 The work of L. Mazet \cite{Mazet} was particularly important to our work because the main theorem in \cite{Mazet} is used in the proof of the \textit{Reduction Lemma} and in the proof of Theorem \ref{estimativa de curvatura mais geral possivel na intro}. Moreover, his work was influential to our study of $\mu$-bubbles in positive weighted bi-Ricci curvature (in Section \ref{section:: mu-bubbles and weighted bi-Ricci curvature}), and in the way we estimate the curvature of CMC hypersurfaces in spectral sense (in Proposition \ref{controle espectral do alfa-bi-Ricci em sec nao negativa} and Proposition \ref{controle do bi-Ricci com peso da CMC estavel em ambiente de curvatura negativa}) by reducing the problem to the study of the positive definiteness of some quadratic forms.

                     To prove Theorem \ref{estimativa de curvatura mais geral possivel na intro} we follow the strategy developed in the work of H. Rosenberg, R. Souam, and E. Toubiana \cite{Rosenberg-Souam-Toubiana}, taking advantage of the recent advances in the classification of strongly stable CMC hypersurfaces in low dimensional Euclidean spaces (see \cite{Chodosh-Li-PrimeiraProva-R4}, \cite{CatinoStable},  \cite{Chodosh-Li-R4-mu-bubbles}, \cite{Chodosh-Li-Minter-Stryker}, \cite{Mazet}, \cite{CHL} and Theorem \ref{quarto teorema principal k > 0}). In fact, we have combined their ideas with those presented in the proof of Lemma 2.4 by O. Chodosh, C. Li, and D. Stryker \cite{Chodosh-Li-Stryker-4mfds} and the Main Theorem by E. Hebey and M. Herzlich  \cite{Hebey-Herzlich} (see Theorem \ref{teorema de convergencia intrinseca}).

                  In \cite{Mazet}, the author shows that if $M$ is a two-sided, complete, strongly stable minimal hypersurface immersed in Euclidean six-space, then after a conformal change of metric on (an open subset of) $M$ one arrives at a Riemannian manifold $N$ with uniformly positive weighted bi-Ricci curvature in \textit{spectral sense} (so to verify the hypotheses of Theorem \ref{mu-bubbles em bi-ricci com peso positivo no sentido espectral}). An interesting observation of the work of J. Chen, H. Hong, and H. Li \cite{CHL} is that this conformal change of metric is unnecessary to reach the same conclusion if the strongly stable CMC hypersurface is non-minimal. This observation is important to our work, because the conformal factor used in \cite{Mazet} depends on the Euclidean geometry, and it is not clear how to adapt it to more general ambient spaces. 

                 The strategy used in \cite{CHL} for their positive answer to do Carmo's question in six-dimensional Euclidean space also proceeds by contradiction, ultimately leading to a contradictory finite-volume conclusion. However, they follow a different route from ours, in particular because they study finite index CMC hypersurfaces directly. Their proof makes interesting use of the work of S. Brendle \cite{BrendleIsop-curvatura-naonegativa} to show that a complete, two-sided, finite index CMC hypersurface immersed in six-dimensional Euclidean space must have finitely many ends. To this end, they establish a Sobolev inequality on such hypersurfaces, again relying on \cite{BrendleIsop-curvatura-naonegativa}. This inequality is used there to bound the volume of a domain by that of a compact region outside it, and then a volume comparison argument for hypersurfaces is presented to yield the desired finite-volume contradiction.
                 
                 Three main difficulties arise when attempting to generalize the work of J. Chen, H. Hong, and H. Li \cite{CHL} to more general ambient spaces. First, we lack results to control the number of ends of a finite index CMC hypersurface. Second, one must estimate the curvature of CMC hypersurfaces in a spectral sense in more general ambient spaces and deduce from it volume bounds for $\mu$-bubbles. Third, the work of S. Brendle \cite{BrendleIsop-curvatura-naonegativa} is not applicable.
                 
                 To overcome these difficulties, we have followed a different strategy. First, we have used L. Mazet's \cite{Mazet} resolution of the stable Bernstein problem in the Euclidean six-space in order to reduce the study of finite index CMC hypersurfaces to the study of strongly stable ones, for which there are tools to control the number of ends. In this process, we also prove curvature bounds for strongly stable CMC hypersurfaces. Second, we have verified the desired curvature estimates in spectral sense (see Propositions \ref{controle espectral do alfa-bi-Ricci em sec nao negativa} and \ref{controle do bi-Ricci com peso da CMC estavel em ambiente de curvatura negativa}) and studied the volume estimates for $\mu$-bubbles in generality (see Theorem \ref{mu-bubbles em bi-ricci com peso positivo no sentido espectral}). Third, we have used a theorem of P. Buser \cite{Buser} (\textit{cf.} \cite{Buser}, Theorem 7.1) to show that the Cheeger constant of the CMC hypersurfaces we study is positive. This gives the desired tool to control the volume of a region by the volume of its boundary, and deduce from the uniform upper bound for the volume of the $\mu$-bubbles an upper bound for the volume of the CMC hypersurface, so to derive a contradiction. 

\subsection{Structure of the paper}
    
                 As we develop the tools to prove Theorems \ref{quarto teorema principal k > 0}, \ref{quarto teorema principal k>-1} and \ref{estimativa de curvatura mais geral possivel na intro} we prove several other results concerning the structure and existence of complete, non-compact, finite index CMC hypersurfaces immersed in six dimensional manifolds under various extra assumptions.

                The main goal of Section \ref{section::compactness arguments} is to prove Theorem \ref{estimativa de curvatura mais geral possivel na intro} and the \textit{Reduction Lemma} alluded to in Section \ref{subsection::Some ideas behind the proofs of the main theorems}. The section is divided into four subsections. In Section \ref{subsection:: A convergence result for non-compact Riemannian manifolds}, we recall a classical result concerning the intrinsic convergence of Riemannian manifolds, with minor adaptations for our purposes. Section \ref{subsection:: Curvature estimates for CMC hypersurfaces} contains the proof of Theorem \ref{estimativa de curvatura mais geral possivel na intro}, along with two direct corollaries to the study of finite index CMC hypersurfaces. In Section \ref{subsection:: A compactness result for non-compact CMC hypersurfaces}, we establish a convergence result for non-compact CMC hypersurfaces, which plays a key role in the proof of the \textit{Reduction Lemma}. Finally, Section \ref{subsection:: Applications} presents two applications: the proof of the \textit{Reduction Lemma} and a result on finite index CMC hypersurfaces properly immersed in asymptotically flat Riemannian manifolds.
                
                The main goal of Section \ref{section::The isoperimetric inequality for finite index CMC hypersurfaces} is to prove Proposition \ref{Garante isoperimetrica}, which will be used in Section \ref{subsubsection:: additional results} to prove Theorem \ref{teorema sobre as de crescimento euclideano de volume} and thus answer a question of do Carmo within a certain class of six-dimensional Riemannian manifolds with non-negative sectional curvature and Euclidean volume growth. Proposition \ref{Garante isoperimetrica} concerns the validity of the isoperimetric inequality on finite index CMC hypersurfaces immersed in these ambient manifolds, and also applies in higher dimensions.
                
                In Section \ref{section::Ends and submanifolds theory} we address the problem of estimating the number of ends of stable CMC hypersurfaces. The main goal is to prove Theorem \ref{Vanishing para 1-forma harmonica}, which generalizes Theorem 0.1 in the work of X. Cheng, L.-F. Cheung, and D. Zhou \cite{Detang-XuCheng-Cheung}. As sketched in Section \ref{subsection::Some ideas behind the proofs of the main theorems}, their result plays a key role in the proof of Theorem \ref{quarto teorema principal k > 0}.  
                
                In Section \ref{section:: mu-bubbles and weighted bi-Ricci curvature} we study $\mu$-bubbles in manifolds with uniformly positive \textit{spectral control} on the $\alpha$-bi-Ricci curvature. The main results of this section are Theorems \ref{mu-bubbles em bi-ricci com peso positivo no sentido espectral} and \ref{Aplicacao à geometria Riemanniana}, which play a key role in the proofs of Theorems \ref{quarto teorema principal k > 0} and \ref{quarto teorema principal k>-1}. 
                
                The main purpose of Section \ref{section:MainResults} is to present the proofs of Theorems \ref{quarto teorema principal k > 0} and \ref{quarto teorema principal k>-1}. This section is divided into two subsections: Section \ref{subsection:: When the ambient manifold has non-negative sectional curvature} deals with ambient manifolds with non-negative sectional curvature, whereas Section \ref{subsection:: When the ambient manifold has a lower bound on the sectional curvature} concerns ambient manifolds whose sectional curvature is bounded from below by a negative constant. Each of these subsections is further divided into two parts: the proofs of Theorems \ref{quarto teorema principal k > 0} and \ref{quarto teorema principal k>-1} are presented in Sections \ref{subsubsection:: proof of theorem A} and \ref{subsubsection:: proof of theorem B}, respectively, while the remaining parts are devoted to additional results on the compactness of finite index CMC hypersurfaces in different classes of six-dimensional ambient spaces. These additional results are Theorems \ref{teorema com hipotese topologia finita em curvatura nao negativa}, \ref{CMC de indice finito em ambiente com curvatura limitada por baixo}, and \ref{teorema sobre as de crescimento euclideano de volume}. The first two concern hypersurfaces with finite topology in a broad class of ambient manifolds.
                    
                The objective of Section \ref{section::Manifolds with uniformly positive curvature conditions} is to further study the case where the ambient manifold satisfies an uniformly positive curvature condition. In this context, we study the compactness of minimal hypersurfaces with finite index and prove Theorem \ref{corolario da extensao do teorema do Catino}. This slightly generalizes Theorem 1.2 by G. Catino, P. Mastrolia and A. Roncoroni \cite{CatinoStable} with different methods. The proof is based in a sharp result of K. Xu, adapted in Theorem \ref{adaptacao do K. Xu dimensional constraints}.
    
                    In the Appendix \ref{section::Lower dimensions}, we show that Theorem \ref{adaptacao do K. Xu dimensional constraints} recovers results from the work of X. Cheng \cite{XuCheng} and Q. Deng \cite{Qintao} on the compactness of complete finite index CMC hypersurfaces immersed in Riemannian manifolds $X^{n+1}$, where $n+1 \in \{4,5\}$. However, this approach does not appear to yield further extensions. We take the opportunity to present results that belong to the current state of the art concerning the study of the compactness of stable and finite index CMC hypersurfaces immersed in hyperbolic spaces of low dimensions. In particular, we give an alternative proof for the result of H. Hong \cite{HongH4} on finite index CMC hypersurfaces in $\mathbb{H}^4$, showing that our $\mu$-bubble approach also works in lower dimensions.

	\subsection*{Acknowledgments} 
    	This work is part of the author's Ph.D. thesis at IMPA. I am grateful to my advisor, Lucas Ambrozio, for his constant encouragement and advice, and to Mario Micallef, Luis Florit, Rafael Montezuma, Ben Sharp, and Detang Zhou for their kind interest in this work. I thank Gilles Carron for a valuable email exchange. I also thank my colleagues Luciano L. Junior and Mateus Spezia for helpful conversations. This study was financed in part by the Coordenação de Aperfeiçoamento de Pessoal de Nível Superior - Brasil (CAPES) – Finance Code 001.

    \section{Convergence results and curvature estimates for CMC hypersurfaces} \label{section::compactness arguments}

    In order to study curvature estimates for CMC hypersurfaces in this section, we first fix the notation and conventions. A constant mean curvature (CMC) hypersurface $M$ immersed in a Riemannian manifold $X$ is a submanifold for which the mean curvature vector has constant length. Recall that $M$ is \textit{strongly stable} if and only if  \eqref{desig de estabilidade} holds for every $\phi \in C^{\infty}_0(M)$. In \eqref{desig de estabilidade}, $Ric(\nu)$ denotes the Ricci curvature of the ambient space computed at the unit normal vector to $M$ denoted by $\nu$, and $A$ denotes the second fundamental form of $M$. We use the convention that $A = \langle \nabla_\text{--} \nu, \text{--}\rangle$ and define $H$ as the trace of $A$, so that the round sphere $\mathbb{S}^n(1) \subset \mathbb{R}^{n+1}$ has constant mean curvature $H=n$ with the choice of unit vector field $\nu$ pointing outwards of the unit ball. Finally, Proposition 2.3 in \cite{Barbosa-Manfredo-Eschenburg} provides the proof of equivalent definitions for the notion of weak stability.

\subsection{A convergence result for non-compact Riemannian manifolds} \label{subsection:: A convergence result for non-compact Riemannian manifolds}

    The following theorem is the main result of this subsection, and concerns the intrinsic convergence of Riemannian manifolds, in the pointed Cheeger-Gromov sense. The proof of the theorem is similar to the proof of the Main Theorem in \cite{Hebey-Herzlich}, and we will indicate the necessary modifications.

    \begin{teo}[\textit{cf.} \cite{Hebey-Herzlich}, Main Theorem] \label{teorema de convergencia intrinseca}
        Let $\big(M_m,g_m\big)_{m \ge 1}$ be a sequence of compact, connected, smooth $n$-dimensional Riemannian manifolds with boundary, and let $(p_m)_{m \ge 1}$ be a sequence of points $p_m \in M_m$.
        
        Suppose that $n \ge 2$ and the following properties are satisfied.

        \begin{enumerate}
            \item $d_{g_m}(p_m,\partial M_m) \to +\infty$.
            \item $\exists \, i_0>0$ such that $ \forall \,R>0$, there exists $m_R \in \mathbb{N}$ for which: if $m \ge m_R$, then 
            \begin{itemize}
                \item $d_{g_m}(p_m, \partial M_m)>R+2i_0$
                \item $\forall \,q \in B(p_m,R+i_0) \subset M_m$ the injectivity radius of $(M_m,g_m)$ at $q$ is bounded from below by $i_0$.
            \end{itemize} 
            \item There exist $k \in \mathbb{N} \cup \{0,+\infty\}$ and positive real numbers $(K_j)_{0 \le j \le k}$ such that $$|\nabla^j Ric_{M_m}| \le K_j$$ for every $0 \le j \le k$ and every $m\ge 1$. 
        \end{enumerate}
        
        Then, for any $\beta \in (0,1)$, there exist a connected non-compact $C^{k+2,\beta}$ $(C^\infty$ if $k = +\infty)$ $n$-dimensional manifold $M_\infty$ without boundary, a point $p_\infty \in M_\infty$ and a $C^{k+1,\beta}_{loc}$ $($possibly not $C^{k+2}$, but $C^\infty$ if $k=+\infty)$ complete Riemannian metric $g_\infty$ in $M_\infty$, such that the following holds, possibly after passing to a subsequence:
        
        Given a precompact open subset $\Omega$ of $M_\infty$, and given $R>0$ such that $\Omega$ is a precompact subset of $B(p_\infty,R)$, there exist $L \in \mathbb{N}$ and a sequence of $C^{k+2,\beta}$ $( C^\infty$ if $k=+\infty)$ embeddings $(\phi_m)_{m \ge L}$ from $B(p_\infty,R)$ onto a neighborhood of $p_m \in M_m$, such that: 
        
        \begin{itemize}
            \item $\phi_m^{-1}(p_m)$ converges to $p_\infty$,
            \item the sequence of $C^{k+1,\beta}$ $( C^\infty$ if $k=+\infty)$ Riemannian metrics $\phi_m^* g_m$ on $\Omega$ converges in the $C^{k+1,\beta}$ $( C^\infty$ if $k=+\infty)$ topology to $g_\infty$.
        \end{itemize}  
    \end{teo}

       Recall that a \textit{harmonic coordinate} $\psi: \Omega \subset (M^n,g) \to \mathbb{R}^n$ is a chart on which every coordinate function $\psi^\theta$ satisfies $\Delta_g \psi^\theta =0$. The following lemma is an important ingredient in the proof of Theorem \ref{teorema de convergencia intrinseca}.

    \begin{lema}[\textit{cf.} \cite{Hebey-Herzlich}, Proposition 12] \label{proposição que constroi o limite}
     Let $\big(M_m,g_m\big)_{m \ge 1}$ be a sequence of compact, connected, smooth $n$-dimensional Riemannian manifolds with boundary, and let $(p_m)_{m \ge 1}$ be a sequence of points $p_m \in M_m$.

        Suppose that $n \ge 2$ and the following conditions are satisfied.

        \begin{enumerate}
            \item $d_{g_m}(p_m,\partial M_m) \to +\infty$.
  \item There exists $\lambda \in \mathbb{R}$ such that $Ric_{(M_m,g_m)} \ge \lambda$, $\forall\, m\in\mathbb{N}$.
            \item There exist real numbers $r>0$, $Q>1$ and $\alpha \in (0,1)$, and an integer $k\ge1$, with the following properties: given $R>r$, there exists $L \in \mathbb{N}$ such that 
            \begin{enumerate}[label=(\alph*)]
                \item If $m \ge L$, then $d(p_m,\partial M_m) > 2R$.
                \item For every sequence of points $(y_m)_{m \ge L}$, satisfying $y_m \in B(p_m,R) \subset M_m$, there are harmonic charts $\psi_m : \Omega_m \to B_0(r)$, where $\Omega_m$ is some open neighborhood of $y_m$ in $M_m \backslash \partial M_m$ and $B_0(r)$ is the Euclidean ball of $\mathbb{R}^n$ with center $0$ and radius $r$. Moreover,
                \begin{itemize}
                    \item these harmonic chart satisfy $Q^{-1} \cdot\delta_{ij} \le ( (\psi^{-1}_m)^*g_m)_{ij} \le Q \cdot \delta_{ij}$ as quadratic forms,
                    \item a subsequence of $((\psi^{-1}_m)^*g_m)$ converges in $C^{k,\alpha}(B_0(r))$.
                \end{itemize}
            \end{enumerate}

        \end{enumerate}

        Then, for any $\beta \in (0,\alpha)$, there exist a connected non-compact $C^{k+2,\beta}$ $n$-dimensional manifold $M_\infty$ without boundary, a point $p_\infty \in M_\infty$ and a $C^{k+1,\beta}_{loc}$ $($possibly not $C^{k+2})$ complete Riemannian metric $g_\infty$ in $M_\infty$, such that the following holds, possibly after passing to a subsequence:
        
        Given a precompact open subset $\Omega$ of $M_\infty$, and given $R>0$ such that $\Omega$ is a precompact subset of $B(p_\infty,R)$, there exist $L \in \mathbb{N}$ and a sequence of $C^{k+2,\beta}$ embeddings $(\phi_m)_{m \ge L}$ from $B(p_\infty,R)$ onto a neighborhood of $p_m \in M_m$, such that: 
        
        \begin{itemize}
            \item $\phi_m^{-1}(p_m)$ converges to $p_\infty$,
            \item the sequence of $C^{k+1,\beta}$ Riemannian metrics $\phi_m^* g_m$ on $\Omega$ converges in the $C^{k+1,\beta}$ topology to $g_\infty$.
        \end{itemize}  
    \end{lema}
    
    \begin{proof}
        The proof follows from an adaptation of the proof of Proposition 12 in \cite{Hebey-Herzlich}. We remark that the main steps of the proof of Proposition 12 in \cite{Hebey-Herzlich} are sketched in the proof of Lemma 2.1 by M. T. Anderson \cite{M.Anderson-compacidade}.
    \end{proof}

    The existence of harmonic coordinates with prescribed control on the radii of the Euclidean ball contained in its image, and on the metric coefficients on these coordinates, is the content of Theorem 6 in \cite{Hebey-Herzlich}. This result will be used to verify the hypotheses of Lemma \ref{proposição que constroi o limite} in the proof of Theorem \ref{teorema de convergencia intrinseca}. For the ease of readability we use the following definition: given a real number $Q>1$ and a coordinate system $(x^1,\dots,x^n)$ of a Riemannian manifold, its Riemannian metric tensor $g$ is $C^{k+1,\alpha}$ \textit{$Q$-controlled} when the components $g_{ij}$ of $g$ satisfy 
    \begin{equation}
        Q^{-1} \cdot \delta_{ij} \le g_{ij} \le Q \cdot \delta_{ij} \quad\text{ as quadratic forms,} \label{Q desigualdade das formas quadraticas}
    \end{equation}    
    \begin{equation} \label{controle sobre a metrica}
        \sum_{1 \le |\beta| \le k+1} \sup_{y \in U} |\partial^\beta g_{ij}(y)| + \sum_{|\beta|=k+1} \sup_{y \neq z} \frac{|\partial^\beta g_{ij}(y) - \partial^\beta g_{ij}(z)|}{d_g(y,z)^\alpha} \le Q.
    \end{equation}

    \begin{proof}[Proof of Theorem \ref{teorema de convergencia intrinseca}]  

        First, we prove the theorem when $k < +\infty$. Fix $ \bar \alpha \in (\beta,1)$.  
        
        We apply Theorem 6 in \cite{Hebey-Herzlich} to find real numbers $Q>1$ and $\tilde r>0$, which depend only on the constants $n$, $\alpha$, $i_0$, $K_j$, $0\le j \le k$, and satisfy the following property. Given $R>r$, where $r:=\tilde r/ \sqrt{Q}$, we may find $L>0$ such that: if $m \ge L$, then  
        \begin{itemize}
            \item $d_{g_m}(p_m,\partial M_m) > 2R$;
            \item $d_{g_m}(p_m, \partial M_m)>R+2i_0$;
            \item For every point $q \in B(p_m,R+i_0) \subset M_m$ the injectivity radius of $(M_m,g_m)$ at $q$ is bounded from below by $i_0$;
            \item Moreover, given any sequence $(y_m)_{m \ge L} \in B(x_m,R) \subset M_m$, there are harmonic charts $\tilde \psi_m : B(y_m,\tilde r) \to U_m \subset \mathbb{R}^n$ on which $g$ is $C^{k+1,\alpha}$ $Q$-controlled.
        \end{itemize}
         Note that $U_m$ contains the ball $B(0,r)$. 
        Then we restrict $\tilde \psi_m$ to define $\psi_m: \Omega_m \to B(0,r)$ as a diffeomorphism. By Arzelà-Ascoli's compactness theorem, we know that for any $\alpha' \in (\beta,\bar \alpha)$, a subsequence of the metrics converges $C^{k+1,\alpha'}$ in any of these charts. Therefore, we can apply Lemma \ref{proposição que constroi o limite} to conclude the proof of Theorem \ref{teorema de convergencia intrinseca} in this case.

        Now, suppose that $k=+\infty$. Fix some integer $k_* > 0$. Arguing as above, we find a harmonic chart $\psi_m: \Omega_m \to B(0,r) \subset \mathbb{R}^n$ of $M_m$ around $y_m$, on which we have $C^{k_*+1,\alpha}$ control over $g$. Since $M_m$ is smooth, and $\psi_m$ is harmonic, we find that $\psi_m$ is a smooth chart. 
        
        Now, by Lemma 11.2.6 in \cite{petersen}, we have
        \begin{equation} \label{EDP para a metrica}
            \frac{1}{2} \Delta g_{ij} + P(g,\partial g)= - Ric_{ij}
        \end{equation}
        \noindent in this chart. Here $P$ is a universal rational expression where the numerator is polynomial in the matrix $g$ and quadratic in $\partial g$ (that is, in the first derivatives of $g$ in this chart), while the denominator depends only on $\sqrt{\det g_{ij}}$. An explicit expression for $P$ is computed in the proof of Lemma 11.2.6 in \cite{petersen}. 
        
        Since $|\nabla^{j} Ric| \le K_j$ for every $0 \le j \le k_*+1$, we conclude that $\Delta g_{ij} \in C^{k_*,\alpha}_{loc}$ in the chart $\psi_m$, uniformly in $m$. A bootstrap argument, using the control on the Ricci curvature and the PDE \eqref{EDP para a metrica}, then shows that for every integer $l\ge 0$, we have $C^{l,\alpha}_{loc}$ control for $g$ in the chart $\psi_m$, uniformly in $m$.

        Given an integer $l \ge 0$, as in the first part of the proof of this theorem, we apply Lemma \ref{proposição que constroi o limite} to obtain a connected, complete, $C^{l+2,\beta}$ pointed manifold $(M_\infty,p_\infty)$ and a $C^{l+1,\beta}_{loc}$ (possibly not $C^{k+2}$) Riemannian metric $g_\infty$ in $M_\infty$, such that the following holds. Given an open and precompact $\Omega \subset \subset M_\infty$, and given $R>0$ such that $\Omega \subset \subset B(p_\infty,R)$, there exists $m_R \in \mathbb{N}$ such that for every $m \ge m_R$, there is a $C^{l+2,\beta}$ embedding $\phi_m : B(p_\infty,R) \to M_m$ onto a neighborhood of $p_m \in M_m$, such that $\phi_m^{-1}(p_m)$ converges to $p_\infty$, and moreover the sequence of $C^{l+1,\beta}$ Riemannian metrics $\phi_m^* g_m$ on $\Omega$ converges in the $C^{l+1,\beta}$ topology to $g_\infty$. 

        Now, we note that, because of the uniform $C^{l,\alpha}_{loc}$ control for $g$ in the chart $\psi_m$, chosen using the fixed $k_*$ as before, the construction of $(M_\infty,g_\infty)$ in the proof of Lemma \ref{proposição que constroi o limite} can be made so that it does not depend on $l$. Moreover, the embeddings defined in the proof of Lemma \ref{proposição que constroi o limite} can be taken smooth, and passing to a subsequence, using a diagonal argument, we can guarantee the desired smooth convergence.
    \end{proof}

\subsection{Curvature estimates for CMC hypersurfaces} \label{subsection:: Curvature estimates for CMC hypersurfaces}

    In this subsection, we prove curvature estimates for stable CMC hypersurfaces immersed in low dimensional manifolds, and derive consequences to the study of the compactness of finite index CMC hypersurfaces. 
    
    The main result of this subsection is Theorem \ref{estimativa de curvatura mais geral possivel na intro}, which we state again as follows:

\begin{teo} \label{estimativa de curvatura mais geral possivel}
                Let $X$ be a complete Riemannian manifold of dimension $3 \le n+1 \le 6$. Suppose that $X$ has bounded sectional curvature $|sec_X| \le K$, for some positive real number $K$. Then there exists a constant $C=C(K) > 0$, that depends only on $K$, with the following property. Any compact, connected, two-sided and strongly stable CMC hypersurface $M$ with nonempty boundary immersed in $X$ satisfies that
                \begin{equation*}
                    |A(p)|\min \{ d_M(p,\partial M), 1 \} \le C, \quad \quad \forall\, p \in M.
                \end{equation*}
    \end{teo}   

    The proof will require the following lemma. 

    \begin{lema} \label{controle na segunda forma controla raio de injetividade v2}
         Let $V = B(0,R)$ be a Euclidean ball of radius $R>0$ in $\mathbb{R}^{n+1}$, and suppose that $g$ is a smooth Riemannian metric on $V$ with the following properties:
        \begin{enumerate}
            \item There exists $K>0$ such that $|sec_{(V,g)}| \le K$;
            \item There exists $Q>1$ such that $g$ is $C^{1,\alpha}$ $Q$-controlled in Euclidean coordinates.
        \end{enumerate}
        
        Let $M$ be a compact smooth hypersurface with nonempty boundary immersed in $V$ such that $|A|_g(p) \le C$, $\forall \, p \in M$, for a constant $C>0$.
        
        Then, given a real number $\delta>0$, there exists a constant $i_0 = i_0(C,Q,K,n,R,\delta)$ with the following property: if $\Omega \subset M$ is an open subset, and $\Omega(2\delta) := \{p \in M: d_M(p,\Omega) < 2\delta\}$ is precompact in $M \backslash \partial M$, then the injectivity radius of $M$ at each point $q \in \Omega$ is bounded from below by $i_0$.  
    \end{lema}

    \begin{proof}
        By Gauss equation, we can bound the sectional curvature of $M$. Using Rauch's theorem, we bound from below, for each $q \in \Omega$, the time for first conjugate point along a unit speed geodesic that starts at $q$. By cut-locus theory, it is enough to bound from below the length of geodesic loops $c : [0,a] \to M$ that are based at some point $q \in \Omega$. 
        
        Every geodesic loop in $M$ is a loop with bounded total curvature in $\mathbb{R}^n$ with respect to the Riemannian metric $g$. Due to the $C^{1,\alpha}$ control on the metric, we can bound the total curvature of these loops with respect to the canonical Euclidean metric by a constant that depends only on $C$, $Q$, $K$, $n$, $R$ and $l_g(c)$. Finally, we will prove that $ \int_c |c''|_{can} ds_{can} \ge \frac{\pi}{2}$, and this will be enough to bound from below $l_g(c)$ by a constant depending only on $C$, $Q$, $K$, $n$ and $R$. 

        We change the notation. Suppose that $\gamma$ is a loop in $M$, parametrized with unit speed with respect to the canonical metric. Let $ \gamma':[0,a] \to \mathbb{S}^{n}$ denote the velocity map of $ \gamma$. 

        Note that there cannot exist $v \in \mathbb{R}^{n+1}$ such that $\langle \gamma',v \rangle > 0$ on $[0,a]$, because
        \begin{equation*}
             \int_0^a \langle  \gamma',v \rangle \,dt = \langle \gamma,v \rangle \big|_0^a =0.
        \end{equation*}

        It follows that $\int_\gamma |\gamma''|ds_{can} \ge \frac{\pi}{2}$. Otherwise, $\gamma'$ would be contained in the hemisphere of $\mathbb{S}^{n}$ determined by $\gamma'(0)$.
    \end{proof}

    \begin{proof}[Proof of Theorem \ref{estimativa de curvatura mais geral possivel}]
       The argument is by contradiction. Suppose that there exists a sequence of complete Riemannian manifolds $(X_i,g_i)$, all with the same dimension $3 \le n+1 \le 6$, all with bounded sectional curvature $|sec|\le K$, and suppose that there exists a sequence of compact, connected, smooth Riemannian $n$-manifolds $M_i$ with nonempty boundary, and isometric immersions $\psi_i:M_i\to X_i$ with constant mean curvature, such that each $M_i$ is two-sided and strongly stable and satisfies
        \begin{equation*}
            |A_i(p_i)| \cdot \min \{ d_{M_i}(p_i,\partial M_i), 1 \} > i
        \end{equation*}
    
        \noindent for some $p_i \in M_i$. Here $i$ denotes a positive integer. 
       
       We will construct a complete, two-sided, non-compact, strongly stable CMC hypersurface immersed in $\mathbb{R}^{n+1}$ which is not an affine hyperplane. When $n+1=6$, this will contradict the main theorems in \cite{Mazet} and \cite{CHL}. When $n+1 \in \{4,5\}$, this will contradict the main theorems in \cite{XuCheng}, \cite{Elbert-Nelli-Rosenberg}, \cite{Chodosh-Li-R4-mu-bubbles}, and \cite{Chodosh-Li-Minter-Stryker}. When $n+1 = 3$, this will contradict the main theorem in \cite{AlexandreDaSilveira} (as in \cite{Rosenberg-Souam-Toubiana}).

        Note that $|A_i(p_i)| \to +\infty$ and $p_i \in M_i \backslash \partial M_i$. We can assume, without loss of generality, by modifying $M_i$ if necessary, that $d_{M_i}(p_i,\partial M_i) \to 0$, $p_i \in M_i \backslash \partial M_i$ maximizes the product $|A_i(q)| \cdot d_{M_i}(q,\partial M_i)$ for $q \in M_i$ and $|A_i(p_i)|\cdot d_{M_i}(p_i,\partial M_i) \to + \infty$ as $i \to \infty$. 

        To simplify the notation, fix an integer $i \ge 0$ and let $(M,p,h) := (M_i,p_i,\psi_i^{*}g_i)$, $\psi := \psi_i$, and $(X,q,g) := (X_i,\psi_i(p_i),g_i)$. By Lemma 2.2 in \cite{Rosenberg-Souam-Toubiana}, the injectivity radius of $\overline{B}(\vec{0}, \frac{\pi}{4\sqrt{K}})$ in $(B(\vec{0},\frac{\pi}{\sqrt{K}}),\exp_q^*g)$ is at least $\frac{\pi}{4\sqrt{K}}$. Note that $(B(\vec{0},\frac{\pi}{\sqrt{K}}),\exp_q^*g)$ has bounded sectional curvature $|sec| \le K$. Fix $\alpha \in (0,1)$. Using Theorem 6 in \cite{Hebey-Herzlich}, we find that there exists real numbers $Q>1$ and $r>0$, such that $Q$ and $r$ depend only on $n$, $\alpha$ and $K$, so that there exists a harmonic chart $(U,\varphi,\mathcal{B}(\vec{0},r))$, where $U$ denotes an open subset of $\mathbb{R}^n$ containing the origin and $\mathcal{B}(\vec{0},r)$ is the geodesic ball in $(B(\vec{0},\frac{\pi}{\sqrt{K}}),\exp_q^*g)$, centered at $\vec{0}_q$ and of radius $r$, with $\varphi(0) = \vec{0}_q$, and such that the metric tensor $\varphi^*\exp_q^*g$ is $C^{1,\alpha}$ $Q$-controlled. We can assume $r< \min \{1,\frac{\pi}{4\sqrt{K}} \}$. Observe that $\exp_q(\mathcal{B}(\vec{0},r)) = B(q,r)$ is the metric ball around $q$ in $X$, of radius $r$.

        Note that $\exp_q:B(\vec{0}_q,r) \subset T_qX \to X$ is a local diffeomorphism and we can assume that $\psi(M)$ is contained in $B(q,r)$. Since $\exp_q :B(\vec{0}_q,r) \to X$ is a local diffeomorphism, both $\psi \times \exp_q: M\times B(\vec{0}_q,r) \to X \times X$ and its restriction to $\partial M \times B(\vec{0}_q,r)$ are transversal to the diagonal $\Delta_{X\times X}$. We define $N:= (\psi \times \exp_q)^{-1}(\Delta)$ and note that $N$ is a smooth manifold with nonempty boundary embedded in $M \times B(\vec{0}_q,r)$, possibly disconnected and non-compact, closed as a subspace of $M \times B(\vec{0}_q,r)$, and we can describe its boundary as $\partial N = N \cap \big(\partial M \times B(\vec{0}_q,r) \big)$. 
        
        Let $\tilde M$ be a connected component of $N$ which has a point $\tilde p$ in its interior such that $\tilde \pi(\tilde p) = p$, where $\tilde \pi : \tilde M \to M$ is the projection in the first coordinate. The map $\tilde \pi$ is a local diffeomorphism. We define $\tilde \psi: \tilde M \to B(\vec{0}_q,r)$ as the projection in the second coordinate, which is an immersion. Note that $\psi \circ \pi = \exp_q \circ \,\tilde \psi$. 
        
        As before, we consider $B(\vec{0}_q,r)$ as a Riemannian manifold with the pullback metric $\exp_q^*g$. We also consider $\tilde \psi$ as an isometric immersion, and $\tilde \pi$ as a local isometry. Is is straightforward to check that $\tilde \psi$ is two-sided and CMC. Lifting to $\tilde M$ a positive function in the kernel of the Jacobi operator of $M$, we check that $\tilde M$ is strongly stable. In this step, we are using standard results from D. Fischer-Colbrie and R. Schoen, as in \cite{Rosenberg-Souam-Toubiana}.

        Note that if $\tilde A$ denotes the second fundamental form of $\tilde \psi$, we have $|\tilde A|(\tilde p) = |A|(p)$.

        Let $\tilde R$ be the supremum of the real numbers $R>0$ such that $B_{\tilde M}(\tilde p,R) \subset \subset \tilde M \backslash \partial \tilde M$. The number $\tilde R$ is well defined and positive, because of the definition of $N$ (and $\tilde M$), and because $\tilde p$ is an interior point of $\tilde M$. 
        
        We claim that $\tilde R \ge d_M(p,\partial M)$. It is enough to prove that every geodesic $\tilde \gamma_v$ of $\tilde M$ that starts at $\tilde p$ with velocity $v$, $|v| = 1$, is defined in the interval $\big[0,d_M(p,\partial M) \big)$. If there exists $v \in T_{\tilde p}\tilde M$, such that $\tilde \gamma_v$ is defined for a maximal time $T<d_M(p,\partial M) < r$, we write $\tilde \gamma_v(t) = (p(t),w(t)) \in M\times \mathcal{B}(\vec{0}_q,T)$ and note that $p(t)$ is a geodesic of $M$ with length $T$ that starts at $p$, and therefore does not touch $\partial M$. That is, $\tilde \gamma_v(t) \in (M \backslash \partial M) \times \overline{\mathcal{B}(\vec{0}_q,T)}$. Using these observations, we can find a sequence $t_j \uparrow T$ such that $\tilde \gamma_v(t_j)$ converges to some point $(u,w) \in (M \backslash \partial M) \times \mathcal{B}(\vec{0}_q,r)$, and since $N$ is a closed subspace of this product space, $(u,w) \in N$. This would guarantee $(u,w) \in \tilde M$, by definition of $\tilde M$ as a connected component of $N$. But then we would be able to extend $\gamma_v$ to $[0,T]$, a contradiction. 

        From now on, we use again the subindex $i$. Note that, by modifying $\tilde M_i$, if necessary, we can assume that $\tilde M_i$ is smooth, connected, compact, with nonempty boundary, and 
        \begin{equation*}
            d_{\tilde M_i}(\tilde p_i, \partial \tilde M_i) \cdot |\tilde A_i|(\tilde p_i) \to +\infty
        \end{equation*}
        
        \noindent when $i \to \infty$. We can also assume that $\tilde p_i$ maximizes the product $d_{\tilde M_i}(\tilde x, \partial \tilde M_i) \cdot |\tilde A_i|(\tilde x)$ among the points $\tilde x \in \tilde M_i$. 
        
        We denote $(Y_i,h_i) := (\mathcal{B}(\vec{0}_{q_i},r),\exp_{q_i}^*g_i)$ and recall that we have a diffeomorphism $\varphi_i: U_i \subset \mathbb{R}^n \to Y_i$, which we view as a chart on which the metric $h_i$ is $C^{1,\alpha}$ $Q$-controlled. Note that the ball of radius $r/\sqrt{Q}$ centered at the origin is contained in $U_i$. 

        Now we proceed with a blow-up argument. 

        Given a point $\tilde x \in B_{\tilde M_i}(\tilde p_i,\rho) \subset \subset \tilde M_i$, we have $d_{\tilde M_i}(\tilde x, \partial \tilde M_i) \ge d_{\tilde M_i}(\tilde p_i, \partial \tilde M_i) - \rho$ and 
        \begin{equation} \label{desigualdade invariante por escala distancia vezes segunda forma}
            d_{\tilde M_i}(\tilde x, \partial \tilde M_i) \cdot |\tilde A_i|(\tilde x)  \le d_{\tilde M_i}(\tilde p_i, \partial \tilde M_i) \cdot |\tilde A_i|(\tilde p_i) 
        \end{equation}

        \noindent hence 
       \begin{equation*}
            |\tilde A_i|(\tilde x)  \le \frac{d_{\tilde M_i}(\tilde p_i, \partial \tilde M_i)}{d_{\tilde M_i}(\tilde p_i, \partial \tilde M_i)-\rho} \cdot |\tilde A_i|(\tilde p_i) \le 2 \cdot  |\tilde A_i|(\tilde p_i)
        \end{equation*}

        \noindent whenever $\rho \le \frac{1}{2}d_{\tilde M_i}(\tilde p_i, \partial \tilde M_i)$.

        Let $\lambda_i := |\tilde A_i|(\tilde p_i)$. Consider the diffeomorphism $D_{\lambda_i}:\mathbb{R}^n \to \mathbb{R}^n$ given by dilation by $\lambda_i$. We denote by the same name the diffeomorphism $D_{\lambda_i}:U_i\to D_{\lambda_i}(U_i)$. Note that $V_i:=D_{\lambda_i}(U_i)$ is an open subset of $\mathbb{R}^n$ containing the ball of radius $\lambda_i \cdot r/\sqrt{Q}$ centered at the origin. Since $\lambda_i \to +\infty$, we assume without loss of generality that the sequence of open sets $V_i$ is nested. 
        
        We are interested in the sequence of immersions 
        \begin{equation*}
            F_i :=D_{\lambda_i}\circ \varphi_i^{-1} \circ \tilde \psi_i : \tilde M_i \to V_i \subset \mathbb{R}^n.
        \end{equation*}

        Note that $F_i(\tilde p_i) = 0$ for every $i \ge 1$. We endow $V_i$ with the Riemannian metric $\langle \cdot,\cdot\rangle_i:=\frac{1}{\lambda_i^2}\big((D_{\lambda_i})^{-1} )^*\varphi^* h_i\big)$ and note that the $C^{1,\alpha}$ control over $\varphi^* h_i$ is enough to guarantee that, possibly after linear changes of coordinates, $\langle \cdot,\cdot\rangle_i$ converges $C^{1,\alpha}_{loc}$ in Euclidean coordinates to the standard Euclidean inner product $can$ (see \cite{Rosenberg-Souam-Toubiana}). 

        From now on, we endow $\tilde M_i$ with the pullback Riemannian metric $F_i^* \langle \cdot,\cdot\rangle_i$. Then $F_i$ is an isommetric immersion of a two-sided, strongly stable CMC hypersurface, $|A_{F_i}|(\tilde p_i) = 1$ and for every $\tilde x \in B_{\tilde M_i}(\tilde p_i,\rho) \subset \subset \tilde M_i$ with $\rho \le \frac{1}{2}d_{\tilde M_i}(\tilde p_i, \partial \tilde M_i)$, $|A_{F_i}|(\tilde x) \le 2$. Moreover, $d_{\tilde M_i}(\tilde p_i, \partial \tilde M_i) \to + \infty$. The last two observations follows from the fact that the inequality \eqref{desigualdade invariante por escala distancia vezes segunda forma} is scale invariant. 

        Using Proposition 4.1 of Appendix 4 in \cite{Rosenberg-Souam-Toubiana} (with obvious modifications to adapt it to higher dimensions) and the $C^{1,\alpha}_{loc}$ convergence of $F_i^*\langle \cdot,\cdot\rangle_i$ to the Euclidean metric we derive: given $R>0$, we can find $i_R \in \mathbb{N}$ such that for every $i > i_R$ we have $B(\tilde p_i, 2R;F_i^* \langle \cdot,\cdot\rangle_i)  \subset \subset \tilde M_i \backslash \partial \tilde M_i$, $B(\tilde p_i, 2R;F_i^*can)  \subset \subset \tilde M_i \backslash \partial \tilde M_i$ and $|A_{(F_i,F_i^*can)}|(\tilde x) \le 5$ for every $\tilde x \in B(\tilde p_i, 2R)$. 

        Now, assume $R>1$. Let $i>i_R$ and $\tilde x \in B(\tilde p_i, R;F_i^*can)$. Let $v \in T_{\tilde x} \tilde M_i$ be a unit vector and $\gamma_v$ the geodesic of $(\tilde M_i,F_i^*can)$ that starts at $\tilde x$ with velocity $v$. Note that $\gamma_v$ is defined on $[-R,R]$ because $B(\tilde x, R;F_i^*can) \subset B(p,2R;F_i^*can) \subset \subset \tilde M_i \backslash \partial \tilde M_i$. Consider $f:[-R,R] \to \mathbb{R}$ given by $f(t) = \langle \gamma_v(t) -\gamma_v(0),v \rangle$, where we have used $\langle\cdot,\cdot \rangle$ for $can$. Computing two derivatives of $f$, and using the Fundamental Theorem of Calculus, we prove $f(\frac{1}{5}) \ge \frac{1}{10}$. This shows that $F_i( \partial B(\tilde x, \frac{1}{5}); F_i^*can) \subset \mathbb{R}^n \backslash B^{\mathbb{R}^n}(F_i(\tilde x), \frac{1}{10})$. Therefore, we can use Lemma 4.1.1 in \cite{Perez-Ros} (adapted to higher dimensions) to write $F_i$ locally as a graph of uniform (and universal) Euclidean size over the tangent plane of each point. 
        
        Using that $F_i : (\tilde M_i, F_i^{*}\langle \cdot,\cdot\rangle_i) \to (V_i,\langle \cdot,\cdot\rangle_i)$ is an isometric immersion of a CMC hypersurface, we see that the functions defining these local graphs satisfy elliptic PDEs: we have the equation 
          \begin{equation*}
            H = \frac{1}{|\nabla F|_g} \Delta F - \frac{1}{|\nabla F|_g^3}HessF(\nabla F, \nabla F)
        \end{equation*}
        \noindent for points of $F^{-1}(0) = graph(u)$, where $F(x_1,\dots,x_n,y) = u(x_1,\dots,x_n)-y$, and $N = \frac{\nabla F}{|\nabla F|_g}$ is the unit normal to $graph(u)$. Therefore, $u$ satisfies a PDE of the form
        \begin{equation*}
            \sum_{1 \le i,j \le n} a^{ij}(g,\partial u) \partial_{ij}u = \Phi(g,\partial g,\partial u) + H f(\partial u)
        \end{equation*}

        \noindent where $\Phi$, $f$, $a^{ij}$ are smooth functions of their entries, and $a = (a^{ij})$ take values in the set of positive definite matrices. The coefficients of this elliptic PDE in non-divergence form are $C^{0,\alpha}$ controlled, because of the $C^{1,\alpha}$ control over the metric coefficients, and the bounds on the second fundamental form of $F_i$. Using Schauder estimates (Corollary 11.2.3 in \cite{petersen}), we derive $C^{2,\alpha}$ control for the function that defines the local graph of this CMC immersion. 

        Using the Gauss equation, we derive uniform local bound for the sectional curvature of $(\tilde M_i,  F_i^{*}\langle \cdot,\cdot\rangle_i)$. Using Lemma \ref{controle na segunda forma controla raio de injetividade v2} we see that the injectivity radius of $(\tilde M_i,  F_i^{*}\langle \cdot,\cdot\rangle_i)$ is locally uniformly bounded. 

        Therefore, we are in position to use Theorem \ref{teorema de convergencia intrinseca} to obtain a $C^{1,\alpha}$ intrinsic limit for the sequence $(\tilde M_i,\tilde p_i, F_i^{*}\langle \cdot,\cdot\rangle_i)$ to a complete, connected, non-compact, pointed $C^{2,\alpha}$ Riemannian manifold $(\tilde M_\infty, \tilde p_\infty, \tilde g_\infty)$, with $\tilde g_{\infty}$ of class $C^{1,\alpha}$, possibly not $C^2$. 

        Pick an exhaustion $\Omega_j$ of $(\tilde M_\infty, \tilde p_\infty, \tilde g_\infty)$ by precompact open metric balls centered at $\tilde p_\infty$. From now on, to simplify notation, we use $(M_m,p_m,g_m) :=(\tilde M_m,\tilde p_m, F_m^*\langle \cdot,\cdot \rangle_m)$, for $m \in \mathbb{N}$, and $(M_\infty, p_\infty, g_\infty):=(\tilde M_\infty, \tilde p_\infty, \tilde g_\infty)$. 

        For each $j$, there is $m_j \in \mathbb{N}$ such that for every $m \ge m_j$, there is a $C^{2,\beta}$ embedding $\phi_{m,j} : \Omega_j \to M_m$ onto a neighborhood of $ p_m \in  M_m$, such that $\phi_m^{-1}(p_m)$ converges to $p_\infty$, and moreover the sequence of $C^{1,\beta}$ Riemannian metrics $\phi_m^* g_m$ on $\Omega$ converges in the $C^{1,\beta}$ topology to $g_\infty$. 

        Fix $0<\beta'<\beta$. Fix $j$, and consider $\psi_{m,j} := F_m \circ \phi_{m,j} : \Omega_j \to V_m \subset \mathbb{R}^n$, a sequence of immersions from $\Omega_j$ to $(\mathbb{R}^n,can)$. We claim that we can pass to a subsequence in $m$, to a limit immersion $\psi_j : \Omega_j \to \mathbb{R}^n$ in the $C^{2,\beta'}_{loc}$ topology. To see this, it is enough to bound the $C^{2,\beta}$ norm of each coordinate of $\psi_{m,j}$ uniformly in $m$. But this follows from the fact that $\psi_{m,j}$ is locally a graph of a $C^{2,\beta}$ function, and $\psi_{m,j}(p_\infty) \to 0$ as $m \to \infty$. 

        Using a diagonal argument, passing to a subsequence, we produce an immersion 
        \begin{equation*}
            \psi: (M_\infty,p_\infty,g_\infty) \to (\mathbb{R}^n,can)
        \end{equation*} 
        
        \noindent such that on each $\Omega_j$, $\psi$ is the limit of $\psi_{j,m}$ in the $C^{2,\beta'}$ topology. It is straightforward to check that $\psi$ is a two-sided CMC isometric immersion with $\psi(p_\infty) = 0$ and $|A_\psi|(p_\infty) = 1$. Using the $C^{1,\beta}$ convergence of the metrics, and the $C^{2,\beta'}$ convergence of the immersions, it is straightforward to prove that $\psi$ is strongly stable. However, as we anticipated, the existence of such $\psi$ is a contradiction. 
    \end{proof}

    \begin{obs} \label{remark sobre estimativa de curvatura}        
        The classification of strongly stable CMC hypersurfaces immersed in Euclidean spaces of dimension $n+1 \in \{3,\dots,6\}$ allows the constant $C(K)$ in the statement of Theorem \ref{estimativa de curvatura mais geral possivel} to depend only on the bound $K>0$ for the sectional curvature of the ambient space. If one uses only the classification of complete, two-sided, strongly stable minimal hypersurfaces, instead of the CMC case, then the same proof applies provided the statement of Theorem \ref{estimativa de curvatura mais geral possivel} is modified as follows. First, we additionally assume that the CMC hypersurfaces have mean curvature bounded by a constant $H_+ >0$, that is, $|H| \le H_+$. Second, the constant $C$ depends on both $K$ and $H_+$, i.e., $C = C(K, H_+)$. The main idea is that, under the condition $|H| \le H_+$, the blow-up argument produces a stable hypersurface that is not only CMC, but in fact minimal. 
    \end{obs}

Next, we describe two applications of Theorem \ref{estimativa de curvatura mais geral possivel} to study finite index CMC hypersurfaces, of independent interest. First, we approach general ambient manifolds with bounded curvature. Then we discuss flat space forms. 

As a first corollary of Theorem \ref{estimativa de curvatura mais geral possivel}, we show that if the ambient manifold $X$ has bounded curvature and low dimension, then every complete finite index CMC hypersurface $M$ immersed in $X$ with sufficiently large mean curvature is necessarily compact. 

\begin{coro} \label{existence da curvatura media critica}
    Let $X$ be a complete Riemannian manifold of dimension $3 \le n+1 \le 6$ and bounded sectional curvature $|sec_X| \le K$, for some $K>0$. Let $C=C(K) > 0$ be as in the statement of Theorem \ref{estimativa de curvatura mais geral possivel}. If $M$ is a compact, connected, two-sided and strongly stable CMC hypersurface with nonempty boundary immersed in $X$ and mean curvature $|H| > \sqrt{n} \cdot C$, then
     \begin{equation*}
            d_M(p,\partial M) \le C \cdot \frac{\sqrt{n}}{|H|}\,, \quad \quad \forall \, p \in M.
        \end{equation*}
\end{coro}

\begin{proof}
    Since $|A|^2 \ge \frac{1}{n}\cdot H^2$, Theorem \ref{estimativa de curvatura mais geral possivel} readily implies that $\frac{|H|}{\sqrt{n}} \cdot \min \{ d_M(p,\partial M),1\} \le C $ for any $p \in M$. Thus, every point $p \in M$ must lie at a distance smaller than $1$ from the boundary of $M$ when $|H| > \sqrt{n} \cdot C$, and the precise estimate above follows.
\end{proof}

We conclude this subsection with an application to flat space forms. 

        \begin{coro} \label{corolario das flat}
             Let $X^{n+1}$ be a flat Riemannian manifold of dimension $3 \le n+1 \le 6$. Every complete, non-compact, non-minimal finite index CMC hypersurface immersed in $X$ is compact.
        \end{coro}

\begin{proof}
       Using Theorem \ref{estimativa de curvatura mais geral possivel} with $K=1$ we find, as in the proof of Corollary \ref{existence da curvatura media critica}, that there exists a universal constant $C$ such that every complete, non-compact, finite index CMC hypersurface immersed in a flat Riemannian manifold has mean curvature $|H|\le C$. The theorem now follows from the fact that rescaling the metric of $X$ we still obtain a flat Riemannian metric, but the mean curvature of non-minimal CMC hypersurfaces scales non-trivially.    
\end{proof}

The above result was known in the range $3 \le n+1 \le 5$; the three-dimensional case follows from the work of R. Lopez and A. Ros \cite{LopezRos}, and in dimensions four and five it follows from the work of X. Cheng \cite{XuCheng}.

\begin{obs} \label{existence da curvatura media critica remark}
Let $H_*$ be the infimum of the set of numbers $\theta \in [0,+\infty]$ with the following property: every complete, finite index CMC hypersurface immersed in $X$ with mean curvature $|H|> \theta$ is compact. It follows from the Corollary \ref{existence da curvatura media critica} that, for Riemannian manifolds of dimension $3 \le n+1 \le 6$, there exists an upper bound for $H_*$ that depends only on a bound $|sec| \le K$, for some $K>0$, and on the dimension. Also, Corollary \ref{corolario das flat} shows that, if $X$ is flat, then $H_* =0$.
\end{obs}

\subsection{A compactness result for non-compact CMC hypersurfaces} \label{subsection:: A compactness result for non-compact CMC hypersurfaces}
    In this subsection, we prove a compactness result that asserts that if the ambient manifold has a well behaved geometry, then out of a sequence of  immersed strongly stable CMC hypersurfaces with uniformly bounded second fundamental form, we may take a subsequence that converges in an appropriate sense to an immersed CMC hypersurface which is also strongly stable. 

    \begin{prop}  \label{Compactness theorem}    
        Let $(X^{n+1},h)$ be a complete Riemannian manifold. Suppose that $X$ has positive injectivity radius, and there are real constants $C(j)>0$, for every integer $j \ge 0$, so that $|\nabla ^j Rm_X|_h \le C(j)$ for all $j \ge 0$, where $Rm_X$ denotes the Riemann tensor of X.
        
        Let $\psi_i : (M_i^n,g_i) \to X$ be a sequence of isometric immersions given by smooth compact, connected, two-sided, constant mean curvature hypersurfaces with nonempty boundary. Suppose that
    \begin{enumerate}
        \item there exists $p_i \in M_i$ such that $\psi_i(p_i)$ converges to a point $x \in X$,
        \item $d_{M_i}(p_i, \partial M_i) \to \infty$,
        \item there exists a real number $L > 0$ such that $|A_i|_{g_i}(p) < L$, $\forall \, p \in M_i$, for every $i \in \mathbb{N}$,
    \end{enumerate}

    \noindent where $A_i$ denotes the second fundamental form of $\psi_i$.

     Then, passing to a subsequence, there exists an isometric immersion $\psi : (M^n,g) \to X^{n+1}$ of a complete, connected, non-compact, smooth Riemannian manifold $(M,g)$ without boundary such that $\psi_i$ converges to $\psi$ in the following sense:
     
     For every compact domain with smooth boundary $\Omega \subset\subset M$, there exist an integer $m>1$ and a sequence of smooth embeddings $F_i : (\Omega,g) \to (M_i,g_i)$, $i \ge m$, such that $F_i^* g_i$ converges in the $C^{1,\alpha}$ topology to $g$ on $\Omega$. Moreover, locally, the composition $\psi_i \circ F_i$ converges to $\psi$ graphically.
     
     In addition, $\psi$ is a two-sided CMC hypersurface immersed in $X$ and the mean curvatures of $M_i$ converge to the mean curvature of $M$. Finally, if every $\psi_i$ is strongly stable, then $\psi$ is strongly stable.       
    \end{prop}

    \begin{proof}
    Note that each $\psi_i$ is locally, that is, in a harmonic coordinate chart of $X$, a graph about a disk of uniform radius on its tangent plane of a function $u$, such that $u$ satisfies a PDE of the form
        \begin{equation*}
            \sum_{1 \le l,j \le n} a^{lj}(h,\partial u) \partial_{lj}u = \Phi(h,\partial h,\partial u) + H_i \cdot f(\partial u)
        \end{equation*}

        \noindent where $\Phi$, $f$ and $a_{lj}$ are smooth functions of their parameters, and $H_i$ is the mean curvature of $M_i$. Moreover, $(a_{lj})$ is positive definite. We have $C^{k,\alpha}$ control for $h$ in harmonic coordinates (\textit{cf.} Theorem 6 in \cite{Hebey-Herzlich}), for every $k \ge 1$, and the bound on the second fundamental form of $\psi_i$ guarantees uniform $C^{1,\alpha}$ control for $u$ and $H_i$. Schauder estimates (Corollary 11.2.3 in \cite{petersen}) then show that we have uniform $C^{k,\alpha}$ estimates for $u$, for every $k \ge 1$. 
        
       Using the Gauss equation and Lemma \ref{controle na segunda forma controla raio de injetividade v2}, we check that the family $(M_i,g_i)$ satisfies the hypotheses of Theorem \ref{teorema de convergencia intrinseca}, and conclude that there exists a connected, complete, smooth pointed manifold $(M,p)$ and a smooth Riemannian metric $g$ on $M$, such that the following holds. Given an open, connected and precompact $\Omega \subset  M_\infty$ containing $p$, there exists $m_\Omega \in \mathbb{N}$ such that for every $m \ge m_\Omega$, there is a smooth embedding $\phi_m : \Omega \to M_m$ onto a neighborhood of $p_m \in M_m$, such that $\phi_m^{-1}(p_m)$ converges to $p$, and moreover the sequence of Riemannian metrics $\phi_m^* g_m$ on $\Omega$ converges smoothly to $g$. 

        Fix such $\Omega$, and the corresponding embedding $F_i : (\Omega,g) \to (M_i,g_i)$. Consider 
        \begin{equation*}
            \xi_i^\Omega := \psi_i \circ F_i : (\Omega,g) \to X.
        \end{equation*}

        Take $R>0$ such that every $\xi^\Omega_i$ has image contained in $B(x,R) \subset X$. We use a construction similar to the one presented in the proof of Lemma \ref{proposição que constroi o limite} to embed $B(x,R)$ in some Euclidean space $(\mathbb{R}^{m},can)$, for some $m \in \mathbb{N}$, gluing harmonic charts. Then we see $\xi_i^\Omega$ as a map from $(\Omega,g)$ to Euclidean space. 
        
        We claim that, passing to a subsequence, $\xi_i^\Omega$ converges to a limit isometric immersion 
        \begin{equation*}
            \psi^\Omega: (\Omega,g) \to X
        \end{equation*}

        \noindent where the convergence is meant via smooth convergence of the coordinates of $\xi^\Omega_i$. To prove the claim, it is enough to bound uniformly the $C^{k,\alpha}$ norm of each coordinate function of $\xi^\Omega_i$, but this follows from the fact that $\xi^\Omega_i$ is locally a graph of a function with uniform bound on its $C^{k,\alpha}$ norm, for every $k \ge 1$, and also because $\xi^\Omega_i(p)$ converges to $x$. 
        
        Now, if $\Omega \subset \Omega'$, then $\psi^{\Omega'}$ extends $\psi^\Omega$, and hence we define in this way an isometric immersion 
        \begin{equation*}
            \psi: (M,g) \to X.
        \end{equation*}
    
        The convergence obtained guarantees that $\psi$ is a two-sided CMC hypersurface immersed in $X$ and the mean curvatures $H_i$ of $\psi_i$ converge to the mean curvature of $\psi$. Finally, if each $\psi_i$ is strongly stable, it is straightforward to prove that $\psi$ is strongly stable.
        
        \end{proof}

    \subsection{Applications} \label{subsection:: Applications}
    
    In this subsection we derive two further applications of the results and methods of this section to study finite index CMC hypersurfaces. Our first result is a key ingredient in the proof of Theorem \ref{quarto teorema principal k > 0}.

    \begin{teo}[\textit{Reduction Lemma}] \label{argumento de redução em homogeneas}
         Let $(X^{n+1},g)$ be a Riemannian manifold of dimension $ 3\le n+1 \le 6$. Suppose that the action on $X$ by its isometry group is cocompact. If there exists a non-compact, complete, finite index CMC hypersurface with mean curvature $H$ immersed in $X$, then there exists also a non-compact, complete, strongly stable CMC hypersurface with the same mean curvature $H$ immersed in $X$.
    \end{teo}

    \begin{proof}
        Let $\psi: (M,h) \to (X,g)$ be an isometric immersion of a non-compact, complete, finite index CMC hypersurface with mean curvature $H$. Using the solution to the stable Bernstein problem in low dimensions  and the curvature estimates it implies (see Remark \ref{remark sobre estimativa de curvatura}), we conclude that $M$ has bounded second fundamental form. Fix $p \in M$. Since $M$ has finite index, we can find a real number $R>0$ such that $M \backslash B_R(p)$ is strongly stable (\textit{cf.} Proposition $1$ in \cite{Fischer-Colbrie}). 
        
        Let $p_j$ be a divergent sequence in $M$. Let $r_j \to \infty$ be such that $U_j := B_{r_j}(p_j)$ is a precompact open subset of $M \backslash B_R(p)$. Out of $U_j$, we construct $V_j$ open and precompact subset of $M \backslash B_R(p)$ with smooth boundary, such that $d(p_j, \partial V_j) \to +\infty$. We denote by $\bar V_j$ the closure of $V_j$ in $M$. 
        
        Let $K$ be a compact subset of $X$ such that every orbit of an element of $X$ by the action of its isometry group intersects $K$. For each $j$ we take an isometry $F_j$ of $(X,g)$ that maps $p_j$ to $K$. We are interested in the sequence of isometric immersions $\phi_j:=F_j \circ \psi: (\bar V_j,h) \to (X,g)$. Passing to a subsequence, we can assume that $\phi_j(p_j)$ converges to a point $x \in X$. 
        
        It is clear that the sequence of immersions $\phi_j$ satisfy the hypotheses of Proposition \ref{Compactness theorem} and thus a subsequence of it converges to an isometric immersion $\phi: (N,g_N) \to X$ of a complete, connected, non-compact, strongly stable CMC hypersurface with mean curvature $H$.  
    \end{proof}
    
    To finish this section, we study finite index CMC hypersurfaces in \textit{asymptotically flat manifolds}, using the tools developed in the proof of Theorem \ref{estimativa de curvatura mais geral possivel}.
    
    A complete Riemannian manifold $(X^n,g)$ is $C^2$-asymptotically flat (with rate $\xi> 0$) if there exist a compact set $K \subset X$ and a diffeomorphism $\varphi^{-1}: X \backslash K \to \{x \in \mathbb{R}^n : |x|> 1\}$ such that
    \begin{equation*}
        g_{ij} = \delta_{ij} + \tau_{ij} \quad \text{and} \quad \partial_\beta \tau_{ij} = O( |x|^{-\xi-|\beta|})
    \end{equation*}
    for all multi-indices $\beta$ of length $|\beta| \le 2$. The map $\varphi$ is referred to as the chart at infinity of $X$. 

    Our next result concerns properly immersed CMC hypersurfaces. An immersion $\psi: M \to X$ is \textit{proper} when the pre-image of every compact subset of $X$ by $\psi$ is a compact subset of $M$.

    \begin{teo} \label{teorema assintoticamente planas}
        Let $(X^{n+1},g)$ be a $C^2$-asymptotically flat Riemannian manifold, with dimension $3 \le n+1 \le 6$. Then every complete, non-minimal, finite index CMC hypersurface properly immersed in $X$ is compact. 
    \end{teo}
    
    \begin{proof}
        On the contrary, suppose that there exists $\psi: M^n \to X$ an isometric immersion of a complete, non-compact, finite index CMC hypersurface, so that $\psi$ is proper and has mean curvature $H>0$. 
        
        Note that $X$ has bounded sectional curvature, because it is asymptotically flat. Let $L>0$ be such that $|sec_X| \le L$ and let $C = C(L)$ be as in Theorem \ref{estimativa de curvatura mais geral possivel}.
        
        Fix $p \in M$. Since $M$ has finite index, we can find $R>0$, so that $M \backslash B_M(p,R)$ is strongly stable. Hence we deduce from Theorem \ref{estimativa de curvatura mais geral possivel} that for every $q \in M \backslash B_M(p,R+2)$, $|A|(q) \le C$. This shows that $M$ has bounded second fundamental form.
        
        Let $p_j$ be a divergent sequence in $M$, and note that $\psi(p_j)$ diverges in $X$ because $\psi$ is proper. Let $r_j \to +\infty$ be such that $B_j:=B_M(p_j,r_j)$ is contained in $M \backslash B_M(p,R)$, and thus is strongly stable. Note that $B_j \subset B_X(\psi(p_j),r_j)$. 
        
        Consider $\varphi$ as the chart of infinity of $X$, as in the definition of Riemannian asymptotically flat manifolds. We define $v_j := \varphi^{-1}(\psi(p_j))$, and remark that $|v_j| \to +\infty$. Note that we can assume $$ B_{\mathbb{R}^{n+1}}(v_j,\frac{1}{2}r_j) \subset \varphi^{-1}(B_X(\psi(p_j),r_j) \subset B_{\mathbb{R}^{n+1}}(v_j,2r_j).$$ 

        For each $j$, we consider the diffeomorphism $T_j: B_{\mathbb{R}^{n+1}}(0,2r_j) \to B_{\mathbb{R}^{n+1}}(v_j,2r_j)$ given by translation. Passing to a subsequence, we can assume that $U_j := T_j^{-1}(\varphi^{-1}(B_X(\psi(p_j),r_j))$ are open nested subsets of $\mathbb{R}^{n+1}$ whose union is the whole Euclidean space. Note that the asymptotically flat condition on $X$ guarantees that $h_{lm} := (T_j^{*}g)_{lm} = \delta_{lm}+ \tilde \tau_{lm}$, where 
        \begin{equation*}
            \partial_\beta \tilde \tau_{lm}= O(|x+v_j|^{-\xi-|\beta|})
        \end{equation*}
        for every multi-indices $\beta$ of length $|\beta| \le 2$. Thus $h_{lm}$ converges on $C^2_{loc}(\mathbb{R}^{n+1})$ to $\delta_{lm}$.

            Each $F_j := T_j^{-1}\circ \varphi^{-1} \circ \psi:(B_j,p_j, \psi^* g) \to (\mathbb{R}^{n+1},0,h)$ defines a pointed isometric immersion. Out of $B_j$ we construct $\Omega_j$ as an open and precompact subset of $B_j$ with smooth boundary containing $p_j$, such that $d_{\psi^*g}(p_j,\partial \Omega_j) \to + \infty$. We restrict $F_j$ to the closure of $\Omega_j$ and note that $F_j: (\Omega_j,\psi^*g) \to (\mathbb{R}^{n+1},h)$ has bounded second fundamental form, say $|A_{(F_j,h)}|_h \le C$ with a constant $C$ that does not depend on $j$. 

            Now, we follow the same arguments of the final part of the proof of Theorem \ref{estimativa de curvatura mais geral possivel}. To be precise, in the same way that we have proved in Theorem \ref{estimativa de curvatura mais geral possivel} that the sequence $F_i:(\tilde M_i, F_i^* \langle \cdot, \cdot \rangle_i) \to (\mathbb{R}^{n+1}, \langle \cdot, \cdot \rangle_i)$ constructed there subconverges to a complete, connected, non-compact, strongly stable CMC hypersurface immersed in $(\mathbb{R}^{n+1},can)$, we argue here that the sequence $F_j: (\Omega_j,p_j,\psi^*g) \to (\mathbb{R}^{n+1},0,h)$ subconverges to a complete, connected, non-compact, strongly stable CMC immersion $\bar \psi: (M_\infty^n,p_\infty,g_\infty) \to (\mathbb{R}^{n+1},0,can)$ with mean curvature $H_\infty = \lim H_{F_j} = H>0$.
        
        The existence of $\bar \psi$ contradicts the main theorem in \cite{AlexandreDaSilveira} and \cite{LopezRos} when $n+1=3$, \cite{XuCheng} and \cite{Elbert-Nelli-Rosenberg} when $n+1 = 4 \text{ or }5$, and \cite{CHL} when $n+1 = 6$.
    \end{proof}

        \section{The isoperimetric inequality for finite index CMC hypersurfaces} \label{section::The isoperimetric inequality for finite index CMC hypersurfaces}

The main goal of this section is to prove Proposition \ref{Garante isoperimetrica}, which concerns the validity of the isoperimetric inequality for finite index CMC hypersurfaces. We first establish some relations between the isoperimetric inequality and different Sobolev inequalities on Riemannian manifolds, which are assumed to be smooth, complete, and without boundary.

Let $(M^n,g)$ be a complete Riemannian manifold. We will say that $M$ \textit{satisfies the isoperimetric inequality} when there exists $C>0$ such that, for every compact smooth domain (that is, a codimension zero submanifold with boundary) $\Omega \subset M$, it holds that
\begin{equation*}
    |\Omega|_g \le C \cdot |\partial \Omega|_g^{\frac{n}{n-1}},
\end{equation*}

\noindent where $| \cdot |_g$ stands for the volume of the corresponding object. 

Now, suppose that $(M^n,g)$ has infinite volume. Let $q \in [1,n)$ be a real number. Following E. Hebey \cite{HebeyBook}, we will say that \textit{the Euclidean-type Sobolev inequality of order $q$ is valid} if there exists a real number $C_q >0$ such that for any $u \in C^\infty_0(M)$,
\begin{equation}
    \big ( \int_M |u|^p dV_g \big)^{q/p} \le C_q \int_M | \nabla u | ^q dV_g \label{sobolev q}
\end{equation}

\noindent where $1/p = 1/q - 1/n$. In short, we say that $(I^{eucl.}_{q,gen})$ is valid when \eqref{sobolev q} holds for every $u \in C^\infty_0(M)$. 

As it is well known, such an inequality holds in the Euclidean space $(\mathbb{R}^n,g_{can})$. The following result is also well known.

		\begin{prop}[\textit{cf.} \cite{schoen-yau-lectures}, Section 3.1]
		        A complete Riemannian manifold satisfies the isoperimetric inequality if and only if $(I^{eucl.}_{1,gen})$ is valid. 
		\end{prop}

On the other hand, there are examples of complete Riemannian manifolds for which $(I^{eucl.}_{2,gen})$ is valid, but the isoperimetric inequality is not satisfied (\textit{cf.} \cite{HebeyBook}, Theorem 8.4). It is also interesting to note that if $(I^{eucl.}_{q,gen})$ holds for some $q \in [1,n)$, then $(I^{eucl.}_{s,gen})$ holds for every $s \in [q,n)$ (\textit{cf.} \cite{HebeyBook}, Lemma 8.1). In this sense, the inequality $(I^{eucl.}_{1,gen})$ is the most restrictive among the inequalities $(I^{eucl.}_{q,gen})$.  

We follow \cite{HebeyBook} to introduce Green's functions and relate them to Sobolev inequalities, and refer also to Section 17 of the book \cite{Livro-Peter-Li} by P. Li for some properties of Green's functions.

For a complete, non-compact, Riemannian manifold $(M,g)$ and a point $x \in M$, consider $\Omega \subset \subset M$ such that $x \in \Omega$ and let $G$ be the solution of
\begin{equation*}
    \begin{split}
        \begin{cases}
                    -\Delta_g G &= \delta_x \,\,\,\ \text{in } \Omega, \\
        G&=0 \,\,\,\,\,\,\ \text{on } \partial \Omega.
        \end{cases}
    \end{split}
\end{equation*}

Set $G_x^\Omega(y) := G(y)$ when $y \in \Omega$ and extend it as zero outside $\Omega$. One has $G_x^\Omega \le G_x^{\Omega'}$ if $\Omega \subset \Omega'$. 

		\begin{prop}
                Set $G_x(y) := sup_{ \{ \Omega : x \in \Omega \}} G_x^\Omega(y)$ for $y \in M$. Then,
            
                \begin{enumerate}
                    \item either $G_x(y) = +\infty$, $\forall y \in M$, or
                    \item $G_x(y) < \infty$, $\forall y \in M \backslash \{x\}$.
                \end{enumerate}
            
                This alternative does not depend on $x$ and, in the second case, $G_x$ is called the positive minimal Green's function of pole $x$.
		\end{prop}

In the first case above, we say that $M$ is \textit{parabolic}; in the second case, the manifold is said to be \textit{non-parabolic}. We return to the study of Sobolev inequalities. The following two theorems are due to G. Carron (see \cite{HebeyBook}, Section $8.1$), and relate the growth of Green's function with the validity of Sobolev inequalities.

		\begin{teo} \label{proposicao que relaciona funcao de green com sobolev com p=2}
    Let $(M^n,g)$ be a smooth, complete Riemannian $n$-manifold of infinite volume, $n \ge 3$. The following two propositions are equivalent:

    \begin{enumerate}
        \item The Euclidean-type generic Sobolev inequality $(I^{eucl.}_{2,gen})$ is valid.
        \item $(M,g)$ is non-parabolic and there exists $K>0$ such that, for any $x \in M$ and any $t>0$,
        \begin{equation}
            Vol_g( \{y \in M : G_x(y) > t \}) \le Kt^{-n/(n-2)} ,  \label{controle de Green da Faber-Krahn}
        \end{equation}

        \noindent where $G_x$ is the positive minimal Green's function of pole $x$.
    \end{enumerate}
		\end{teo}                                              

		\begin{teo} \label{Carron Thm}
    If $(M^n,g)$ is a non-parabolic, complete Riemannian manifold whose Ricci curvature is bounded from below, and if there exists $K>0$ such that for any $x \in M$ and any $t>0$ the positive minimal Green's function $G_x$ of pole $x$ satisfies
    \begin{equation}
        Vol_g( \{ y \in M : G_x(y) > t\} ) \le K t^{-n/(n-1)}  \label{controle restritivo de Green}
    \end{equation}

    \noindent then the Euclidean-type generic Sobolev inequality $(I^{eucl.}_{1,gen})$ is valid. 
		\end{teo}      
        
        It is interesting to note that Theorem \ref{Carron Thm} is not sharp. Indeed, the condition above is not satisfied by the positive Green's function $G_x$ of the Euclidean space $\mathbb{R}^n$. Nevertheless, we will see that under adequate geometric constraints on the ambient manifold, we can verify these sufficient conditions for a non-minimal finite index CMC hypersurface. Pointing in that direction, our first proposition relates the growth of Green's functions with the condition $\lambda_1(M)>0$. Recall that for a complete Riemannian manifold $M$ with infinite volume
    \begin{equation} \label{definicao do lambda1}
        \lambda_1(M) = \inf_{0 \neq f \in C^\infty_0(M)} \frac{\int_M |\nabla f|^2}{\int_M f^2}.
    \end{equation}

		\begin{prop}  \label{controle de Green do lambda0}
            Let $(M,g)$ be a complete Riemannian manifold of infinite volume. If $\lambda_1(M) > 0$, then $M$ is non-parabolic and for any $x \in M$ and any $t>0$, the positive minimal Green's function $G_x$ of pole $x$ satisfies
            \begin{equation*}
                Vol_g( \{ y \in M : G_x(y) > t\} ) \le \lambda_1(M) \cdot  t^{-1}.
            \end{equation*}
		\end{prop}

 \begin{proof}
     Let $\Phi_t(y) := \min(G_x^\Omega(y),t)$. Then, $|\nabla \Phi_t| = |\nabla G_x^\Omega| \mathbbm{1}_{ \{G_x < t\}}$. Using $\Phi_t$ as test function for $\lambda_1(M)$ we get 
     \begin{equation*}
         Vol_g \big ( \{y \in M : G_x^\Omega(y) > t \}  \big) t^2 \le  \int_M \Phi_t^2 \le \lambda_1(M) \int_M |\nabla \Phi_t|^2.
     \end{equation*}

     Now, 
     \begin{equation*}
         \int_M |\nabla \Phi_t|^2 = \int_{ \{ G_x^\Omega < t \} } |\nabla G_x^\Omega|^2. 
     \end{equation*}

     For almost every positive $t \in \mathbb{R}$, $\{ G_x^\Omega  < t \} $ has smooth boundary $\{ G_x^\Omega  = t \}$. For these numbers $t$, we integrate by parts to get
     \begin{equation*}
         \int_M |\nabla \Phi_t|^2 = -\int_{ \{ G_x^\Omega < t \} }  G_x^\Omega \Delta G_x^\Omega + t \int_{ \{ G_x^\Omega = t\}} \frac{\partial G_x ^\Omega}{\partial \nu} = t
     \end{equation*}

    \noindent because $\Delta G_x^\Omega = 0$ on ${ \{ G_x^\Omega < t \} }$, and we have used the definition of Green's function. Therefore,  
    \begin{equation*}
        Vol_g( \{y \in M : G_x^\Omega(y) > t \} \le \lambda_1(M) \cdot t^{-1}
    \end{equation*}

    \noindent for almost every $t>0$. Since the left-hand side is monotone in $t >0$ and the right-hand side is continuous, one can prove that the inequality must hold for every $t>0$.

    We note that $G_x^\Omega$ converges pointwise to $G_x$, and the convergence is monotone for increasing $\Omega$. Therefore, $M$ is non-parabolic, and we use the Monotone Convergence Theorem to show that
    \begin{equation*}
        Vol_g( \{y \in M : G_x(y) > t \} \le \lambda_1(M) \cdot t^{-1}
    \end{equation*}

    \noindent as desired.
     \end{proof}

        When $(N^n,g)$ is a non-compact Riemannian manifold with non-negative sectional curvature, it follows from the Bishop-Gromov volume comparison theorem that the limit
         \begin{equation} \label{crescimento euclideano de volume}
                \lim_{R \to \infty} \frac{V(B_R(p))}{R^n}
         \end{equation}
        \noindent exists for every $p \in N$, and assumes a value in the interval $\Big[0,vol(\mathbb{S}^n,can)\Big]$. We say that $N$ has \textit{Euclidean volume growth} when this limit is positive for every $p \in N$. The Sobolev inequalities for these manifolds and its submanifolds have been studied in \cite{BrendleIsop-curvatura-naonegativa} and \cite{CHL}. We will study manifolds with a further geometric property: bounded geometry. A Riemannian manifold has \textit{bounded geometry} when its injectivity radius is positive and its sectional curvature is bounded. (Beware that different authors use different definitions for Riemannian manifolds of bounded geometry.)

       Due to the work of J. H. Michael, and L. M. Simon \cite{Michael-Simon-Sobolev}, and S. Brendle \cite{BrendleIsop-curvatura-naonegativa}, every minimal hypersurface immersed in a Riemannian manifold with non-negative sectional curvature and Euclidean volume growth satisfies the isoperimetric inequality. One loses this property when the mean curvature of the hypersurfaces are constant but not zero, as the case of the right circular cylinder in Euclidean three space shows. On the other hand, we have

		\begin{prop}  \label{Garante isoperimetrica}      
                Let $X^{n+1}$ be a Riemannian manifold with non-negative sectional curvature and bounded geometry. Suppose that $X$ has Euclidean volume growth and dimension $n+1 \ge 6$. If there exists a complete, non-compact, non-minimal CMC hypersurface $M$ immersed in $X$ with finite index and bounded second fundamental form, then $M$ satisfies the isoperimetric inequality.  
		\end{prop}

         We remark that the minimal case could be included in the Proposition \ref{Garante isoperimetrica}, as this case is part of the work of \cite{BrendleIsop-curvatura-naonegativa}. We did not include the case $n+1 =4$ or $5$ in the theorem, because in this case the existence hypothesis never holds (\cite{XuCheng}, \cite{Elbert-Nelli-Rosenberg}). The bound on the second fundamental form of $M$ can be dropped from the assumptions when $n+1=6$ due to a blow-up argument, using the fact that $X$ has bounded curvature and the solution to the stable Bernstein problem due to \cite{Mazet} and the curvature estimate it implies (see Remark \ref{remark sobre estimativa de curvatura}). 
        
        In the proof of the proposition, we will need the following two results. The first one will guarantee that the CMC hypersurface under consideration has infinite volume.

	\begin{teo}[\textit{cf.} \cite{K.Frensel--long}, Corollary 8] \label{Katia Frensel volume infinito}
               In a complete Riemannian manifold with bounded geometry, every complete non-compact CMC immersion has infinite volume.
           \end{teo}

    The second result needed concerns the validity of the condition $\lambda_1(M)>0$ (see \eqref{definicao do lambda1}) for finite index CMC hypersurfaces. 

        \begin{teo} \label{CMC com H>Hlambda tem lambda1>0}    
            Let $X^{n+1}$ be a Riemannian manifold with bounded geometry with dimension $n+1\ge 4$. Let $M$ be a finite index CMC hypersurface immersed in $X$ with mean curvature $H$. If    
            \begin{equation*}
                \frac{1}{n} \cdot H^2+\inf_M Ric_X >0
            \end{equation*}
        
            \noindent then $$\lambda_1(M)>0.$$
        \end{teo}
    
        \begin{proof}
             We use Theorem \ref{Katia Frensel volume infinito} to guarantee that $M$ has infinite volume. If $M$ is strongly stable, then the result follows directly from a manipulation of the strong stability inequality. If $M$ has finite index, then it is strongly stable out of a compact subset $K$ of $M$. Thus, there exists $\delta>0$, such that 
             \begin{equation*}
                 \int_M |\nabla \varphi|^2 \ge \delta \int_M \varphi^2
             \end{equation*}
             \noindent for every $\varphi \in C^\infty_0(M \backslash K)$. With this and the fact that $M$ has infinite volume, one can prove that $\lambda_1(M)>0$. (This result is claimed in page 19 of the lecture notes \cite{Carron-L2-Harmonic-forms}. A particular case can be found in the proof of Lemma 3.4 in \cite{HongH4}, and the computations can be generalized.)
        \end{proof}

\begin{proof}[Proof of Proposition \ref{Garante isoperimetrica}]
     By Theorem \ref{Katia Frensel volume infinito}, $M$ has infinite volume. Since $M$ has constant mean curvature $H>0$ and finite index, we can use Theorem \ref{CMC com H>Hlambda tem lambda1>0} to conclude that $\lambda_1(M)>0$. Fix $x \in M$ and let $G_x$ denote the positive minimal Green's function $G_x$ of pole $x$. It follows from Proposition \ref{controle de Green do lambda0} that $Vol_g( \{ G_x> \frac{1}{R} \}) \le \lambda_1(M) \cdot R$ for every $R>0$. Now, for $R \ge 1$ we have $R \le R^{\frac{n}{n-1}}$ and $Vol_g( \{ G_x > \frac{1}{R} \}) \le \lambda_1(M) \cdot R \le \lambda_1(M) \cdot R^{\frac{n}{n-1}}$.

    On the other hand, by Proposition 2.3 in \cite{CHL}, the Euclidean-type generic Sobolev inequality $(I^{eucl.}_{2,gen})$ is valid on $M$. Hence, it follows from inequality \eqref{controle de Green da Faber-Krahn} in Theorem \ref{proposicao que relaciona funcao de green com sobolev com p=2} that, for some constant $\tilde C>0$, 
    \begin{equation*}
          Vol_g( \{ G> \frac{1}{R} \}) \le \tilde C R^{\frac{n}{n-2}} \le \tilde C R^{\frac{n}{n-1}}
     \end{equation*}
    \noindent for every $R \le 1$.

     Therefore, for every $R>0$ we have $Vol_g( \{ G_x> \frac{1}{R} \}) \le C R^{\frac{n}{n-1}}$. By Gauss equation, $M$ has Ricci curvature bounded from below. Thus we can apply Theorem \ref{Carron Thm} to conclude that $M$ satisfies the isoperimetric inequality. 
\end{proof}

    \section{Estimates for the number of ends of stable CMC hypersurfaces} \label{section::Ends and submanifolds theory}

    This section is devoted to the study of estimates for the number of ends of complete, two-sided stable CMC hypersurfaces. This topic of study is directly related to the theory of harmonic functions on complete manifolds, and this connection was investigated in the work of P. Li and L.-F. Tam \cite{Li-e-Tam-Fins-e-Harmonicas}. The strategy we follow here is related to the seminal work of R. Schoen and S.-T. Yau \cite{schoen-yau-harmonicas}.
    
    The main goal of this section is to prove two theorems regarding the number of ends of stable CMC hypersurfaces immersed in Riemannian manifolds with non-negative $\alpha$-bi-Ricci curvature. Given a real number $\alpha >0$ the $\alpha$-bi-Ricci curvature of a Riemannian manifold $N^{n+1}$ of dimension $n+1 \ge 3$, here denoted by $BRic_\alpha^N$, is defined as $$BRic_\alpha^N(u,v) = Ric^N(u) + \alpha \cdot (Ric^N(v) - sec_N(u \land v))$$ for orthonormal vectors $u$ and $v$. When $\alpha = 1$, this curvature quantity is the bi-Ricci curvature $(BRic^N)$ introduced by Shen and Ye \cite{shen-ye}. Some authors refer to the $\alpha$-bi-Ricci curvature as a weighted bi-Ricci curvature with weight $\alpha$.

The uniformly positive $\alpha$-bi-Ricci curvature condition interpolates, in some sense, between the sectional curvature and the scalar curvature uniformly positive conditions. The reader is referred to the recent work of S. Brendle, S. Hirsch and F. Johne \cite{brendle-intermediate}, and the introduction of the work of O. Chodosh, C. Li and D. Stryker \cite{Chodosh-Li-Stryker-4mfds} for motivations about intermediate curvature conditions.

We highlight that a Riemannian manifold of non-negative (resp. uniformly positive) sectional curvature has non-negative (resp. uniformly positive) $\alpha$-bi-Ricci curvature for any $\alpha>0$. Moreover, a Riemannian manifold with non-negative (resp. uniformly positive) bi-Ricci curvature has non-negative (resp. uniformly positive) scalar curvature. 
    
    Concretely, $BRic_\alpha(u,v) = (1-\alpha)Ric(u) + \alpha \cdot BRic(u,v)$. We highlight the following consequence of this observation. Let $N^{n+1}$ be a Riemannian manifold of dimension $n+1 \ge 3$. If $N$ has non-negative Ricci curvature and non-negative bi-Ricci curvature, then $N$ has non-negative $\alpha$-bi-Ricci curvature for any $\alpha \in (0,1)$. Similarly, a uniformly positive $\alpha$-bi-Ricci curvature condition appears when we combine a non-negative curvature condition with a uniformly positive curvature condition on a Riemannian manifold, and this is related to the recent works (\cite{Chodosh-Li-Stryker-4mfds},\cite{CatinoStable}).

    We refer the reader to \cite{Li-e-Tam-Fins-e-Harmonicas} for the definition of a non-parabolic end of a complete Riemannian manifold, a notion needed to state the next theorem:

\begin{teo} \label{Vanishing para 1-forma harmonica}
    Let $X^{n+1}$ be a complete Riemannian manifold of dimension $3 \le n+1 \le 7$ and bounded geometry. Suppose that there exists a weight
    \begin{equation*}
        \alpha \in \Big(\frac{n-1}{n},\frac{2}{\sqrt{n-1}}\Big]
    \end{equation*}

    \noindent such that $X$ has non-negative $\alpha$-bi-Ricci curvature. Let $M$ be a complete, non-compact, two-sided CMC hypersurface immersed in $X$. Then the following holds.
    
    \begin{enumerate}
        \item If $M$ is strongly stable, then every harmonic $1$-form on $M$ with bounded $L^2$ energy vanishes identically.
        \item If $M$ is weakly stable, then every bounded harmonic function $f \in C^\infty(M)$ with finite Dirichlet energy is constant.
    \end{enumerate}  
    
    In any case, $M$ has at most one non-parabolic end.  
\end{teo}

The particular case of Theorem \ref{Vanishing para 1-forma harmonica} of Riemannian manifolds $X$ with non-negative sectional curvature and dimension $n+1 \le 6$ was proved in Theorem 0.1 by X. Cheng, L.-F. Cheung, and D. Zhou \cite{Detang-XuCheng-Cheung} (see also \cite{Fu-Li}). The proof of Theorem \ref{Vanishing para 1-forma harmonica} will rely on the following lemma from linear algebra, which can be deduced from the proof of Proposition 2.2 in the work of J. Chen, H. Hong and H. Li \cite{CHL}.

\begin{lema}[\textit{cf.} \cite{CHL}, Proposition 2.2] \label{Linear algebra}
    For any symmetric $n \times n$ real matrix $T$, 
    \begin{equation*}
        || tr(T) \cdot  T - T \circ T ||   \le \frac{\sqrt{n-1}}{2} |T|^2
    \end{equation*}

    \noindent where $||\cdot||$ denotes the operator norm, and $|T|$ denotes the Frobenius norm of $T$: $|T|^2 = tr(T \circ T^t)$. Moreover, the inequality is sharp. 
\end{lema}

\begin{proof}[Proof of Theorem \ref{Vanishing para 1-forma harmonica}]

We now prove the first item of the theorem. Under the hypotheses of the proposition, suppose that $\omega$ is a harmonic $1$-form on $M$ with bounded $L^2$ energy. Let $\alpha$ be as in the hypotheses, and let $Z$ be dual to $\omega$, \textit{i.e.} $\omega = \langle Z, \text{--}\rangle$. We begin with Bochner's identity 
    \begin{equation}
        |Z|  \Delta |Z| \ge Ric(Z,Z) + |\nabla Z|^2 - |\nabla |Z||^2 . \label{Bochner}
    \end{equation}

    Now, we recall Kato's improved inequality, using that $\omega$ is harmonic,
    \begin{equation}
        |\nabla Z|^2 \ge  \frac{n}{n-1} |\nabla |Z||^2. \label{Kato}
    \end{equation}

    Therefore, combining \eqref{Bochner} with \eqref{Kato} we have
    \begin{equation}
        |Z|  \Delta |Z| \ge  Ric_M(Z,Z) + \frac{1}{n-1} |\nabla |Z||^2. \label{Bochner+Kato}
    \end{equation}
    
   Integrating \eqref{Bochner+Kato} multiplied by a test function $\phi^2$, and integrating by parts the left-hand side, we obtain
    \begin{equation}
        - \int_M 2\phi|Z|\langle \nabla \phi, \nabla |Z|  \rangle   \ge \int_M  \frac{n}{n-1} |\nabla |Z||^2\phi^2 + Ric_M(Z,Z)\phi^2.  \label{conseq of Bochner}
    \end{equation}
    
    Using $\psi = |Z|\phi$ as a test function in the strong stability inequality, for some $\phi \in C^{\infty}_0$, we obtain
    \begin{equation}
        \int_M  2 \phi|Z| \langle \nabla \phi,\nabla|Z| \rangle + |Z|^2 |\nabla \phi|^2 \ge \int_
        M (|A|^2+Ric_X(\nu,\nu))|Z|^2 \phi^2 - |\nabla|Z||^2 \phi^2. \label{conseq de estabilidade}
    \end{equation}

    We may rewrite \eqref{conseq of Bochner} using the Gauss equation:
    \begin{equation*}
        Ric_M(Z,Z) =  \big (Ric_N(Z,Z) - |Z|^2 sec_X(Z \land \nu) \big ) +  \big (H \langle AZ,Z \rangle - |AZ|^2  \big ).     
    \end{equation*}

    Summing \eqref{conseq of Bochner} with $\varepsilon>0$ times \eqref{conseq de estabilidade}, where $\varepsilon : = \frac{1}{\alpha} \in [\frac{\sqrt{n-1}}{2},\frac{n}{n-1})$, we have
    \begin{equation}
         \int_M \Theta +  \varepsilon |Z|^2|\nabla \phi|^2   \ge \int_M  C(n, \varepsilon) |\nabla |Z||^2\phi^2 + \zeta\cdot \phi^2\label{desiguadade tipo Schoen-Yau}
    \end{equation}

    \noindent where $C(n,\varepsilon) = \frac{n}{n-1}-\varepsilon$, $\Theta = 2( \varepsilon-1)\phi|Z|\langle \nabla \phi, \nabla |Z|  \rangle $ and
    \begin{equation*}
    \begin{split}
        \zeta &= \big (Ric_X(Z,Z) - |Z|^2 sec_X(Z \land \nu) + \varepsilon Ric_X(\nu,\nu)|Z|^2\big )  \\
             &\quad\quad+ \big (H \langle AZ,Z \rangle - |AZ|^2 + \varepsilon |A|^2|Z|^2 \big ).
    \end{split}
    \end{equation*}

     Using the hypothesis $BRic_\alpha^X \ge0$ and Lemma \ref{Linear algebra} we have
    \begin{equation*}
        \zeta \ge 
  (\varepsilon -\frac{\sqrt{n-1}}{2}) |A|^2|Z|^2.
    \end{equation*}

    Note that $C(n,\varepsilon) = \frac{n}{n-1}-\varepsilon > 0$ and $\varepsilon -\frac{\sqrt{n-1}}{2} \ge 0$. We use Cauchy-Schwarz and Young's inequality on $\Theta$ to get, from \eqref{desiguadade tipo Schoen-Yau}, that there exists $\delta >0$ such that
    \begin{equation*}
        \delta \int_M  |Z|^2|\nabla \phi|^2   \ge \int_M   |\nabla |Z||^2\phi^2.
    \end{equation*}      
    
    Choosing $\phi$ such that $\phi \equiv 1$ on $B_R(M)$ and $|\nabla \phi| \le 1/R$ we conclude that $|Z|$ is constant letting $R \to \infty$. Since $X$ has bounded geometry, $M$ has infinite volume by Theorem \ref{Katia Frensel volume infinito}. But $\omega$ is in $L^2$, hence $\omega =0$.  This completes the proof of the first item.

        The proof of the second item follows closely the proof of the first item, except that the admissible test functions $\psi$ for the stability operator are required to satisfy $\int_M \psi =0$. 

    We will use a strategy developed in the work of X. Cheng, L.-F. Cheung, and D. Zhou \cite{Detang-XuCheng-Cheung}. Suppose that $f$ is not constant and pick $p \in M$ such that $|\nabla f|_p >0$. Then, as in the proof of Theorem 3.1 in \cite{Detang-XuCheng-Cheung}, for every $a>1$ and every $R_0 >a$, there exists $R>R_0$ and a Lipschitz function $\eta = \eta_{a,R}$ with compact support in $M$ such that: $\eta \equiv 1$ on $B_p(a)$, $|\nabla \eta| \le 1/R$ and $\int_M |\nabla f|\eta =0$.

    We are now ready to follow the strategy of the proof of the first item of this theorem. The only modification in the proof is that we use the constructed test function $\psi =\eta |\nabla f|$ in the weak stability inequality, instead of using the test function $\psi = |Z|\phi$ in the strong stability inequality there. Therefore, we get from the Bochner's identity, Gauss formula, Lemma \ref{Linear algebra} and the weak stability inequality, the following inequality:
    \begin{equation*}
        \delta \int_M  | \nabla f |^2|\nabla \eta|^2   \ge \int_M   |\nabla | \nabla f||^2\eta^2.
    \end{equation*}   

    \noindent for some $\delta > 0$. Using the properties of $\eta$ we get
    \begin{equation*}
        \frac{\delta}{R^2} \int_M |\nabla f|^2 \ge \int_{B_p(a)} |\nabla |\nabla f||^2.
    \end{equation*}

    Since $R>a$ may be chosen arbitrarily large, we conclude that $|\nabla f|$ is constant. But $\int_M |\nabla f|^2 < +\infty$ and $M$ has infinite volume by Theorem \ref{Katia Frensel volume infinito}. Therefore, $|\nabla f| \equiv 0$, and this contradiction proves that $f$ is constant.

    Finally, we apply Theorem 2.1 in \cite{Li-e-Tam-Fins-e-Harmonicas} to bound the number of non-parabolic ends of $M$. 
\end{proof}

    \begin{obs}
            Under the notation of Theorem \ref{Vanishing para 1-forma harmonica}, we remark that Proposition 2.2 in \cite{Detang-XuCheng-Cheung} shows that if $\inf_M Ric_X > -\frac{1}{n} \cdot H^2$, then each end of $M$ is necessarily non-parabolic.
    \end{obs} 

     We end this section with a result that will be needed in Section \ref{section:MainResults} to prove Theorem \ref{quarto teorema principal k>-1}.

		\begin{teo}[\textit{cf.} \cite{Fu-Li}, Theorem 1.1] \label{unicidade de fim quando curvature é limitada inferiormente em X6}
    Let $X$ be a complete six-dimensional Riemannian manifold with bounded geometry. Suppose that $X$ has sectional curvature bounded from below, $sec_X \ge -1$. Let $M$ be a complete, non-compact weakly stable CMC hypersurface with $|H|=5+\varepsilon$ immersed in $X$. If $\varepsilon \ge 10\sqrt{\frac{5}{11}}-5$, then $M$ has only one end.
		\end{teo}

        \begin{proof}
            By Proposition $2.2$ in \cite{Detang-XuCheng-Cheung}, every end of $M$ is non-parabolic. Moreover, Theorem $1.1$ in the work of H. P. Fu and Z. Q. Li \cite{Fu-Li} guarantees that $M$ has only one non-parabolic end.
        \end{proof}

        The number $10\sqrt{\frac{5}{11}}-5$ is approximately $1.742$, rounded to three decimal places.

        \begin{obs} \label{remark sobre os fins das cmc no hiperbolico}
            Theorem \ref{unicidade de fim quando curvature é limitada inferiormente em X6} will be used to prove Corollary \ref{corolario de classificacao no espaco hiperbolico}. We note that when $X$ is the hyperbolic space $\mathbb{H}^6$, we believe that the lower bound on $\varepsilon$ imposed in Theorem \ref{unicidade de fim quando curvature é limitada inferiormente em X6} is not optimal, but we did not try to optimize this lower bound here. Nevertheless, it is unclear to us if the available tools can show that every complete strongly stable CMC hypersurface immersed in the hyperbolic six-space with mean curvature $|H|>5$ has a finite number of ends.
        \end{obs}

        \section{ \texorpdfstring{$\mu$}{mu}-bubbles in positive weighted bi-Ricci curvature} \label{section:: mu-bubbles and weighted bi-Ricci curvature}

        This section is devoted to the study of volume estimates for $\mu$-bubbles embedded in manifolds with uniformly positive $\alpha$-bi-Ricci curvature in \textit{spectral sense}. This curvature condition is naturally connected with the study of stable CMC hypersurfaces, as recently noted by L. Mazet \cite{Mazet} in the minimal case. In the first part of the section, we prove these volume estimates. We end the section with an application which restricts the geometry of Riemannian manifolds with uniformly positive $\alpha$-bi-Ricci curvature. 

        The precise meaning for positive $\alpha$-bi-Ricci curvature in \textit{spectral sense} will become clear in the statement of the main result of this section: Theorem \ref{mu-bubbles em bi-ricci com peso positivo no sentido espectral}. This curvature condition regards the spectrum of an elliptic operator, and is immediately verified when the Riemannian manifold admits a positive uniform lower bound for its $\alpha$-bi-Ricci curvature; thus the spectral condition is weaker. We will see in Propositions \ref{controle espectral do alfa-bi-Ricci em sec nao negativa} and \ref{controle do bi-Ricci com peso da CMC estavel em ambiente de curvatura negativa} that this weaker curvature condition is verified by some strongly stable CMC hypersurfaces immersed in certain six-dimensional manifolds.  \\

    The $\mu$-bubbles were introduced by M. Gromov in the theory of positive scalar curvature (\textit{cf.} \cite{Four-Lectures-Gromov}, Section 5). Originally, they were conceived as stable critical points of a functional that prescribes the mean curvature of hypersurfaces by a function $\mu$. This tool and modifications of it have found ingenious applications in the theory of stable CMC hypersurfaces since the seminal work of O. Chodosh and C. Li \cite{Chodosh-Li-R4-mu-bubbles}. In the next theorem we collect results from the literature and use them to define the $\mu$-bubbles considered here. 

\begin{teo} \label{definicao de mu-bubble}
    Let $(N,g)$ be a complete Riemannian manifold of dimension $3 \le n \le 7$. Let $Y$ be a precompact open set of $N$ with smooth boundary. Suppose that $\partial Y = \partial_+ Y \sqcup \partial_-Y$ with $\partial_+Y$ and $\partial_-Y$ nonempty and disjoint. 
    
    Given
    
    \begin{enumerate}
        \item a smooth function $h:Y \to \mathbb{R}$ such that $h \to \pm \infty$ on $\partial_\pm Y$,
        \item a real number $a>0$,
        \item  and a positive smooth function $w \in C^\infty(N)$,
    \end{enumerate}  
    
    \noindent there exists a relatively open set $\Omega_*$ in the closure of $Y$ containing an open neighborhood of $\partial_+Y$, where $\Omega_*$ has smooth boundary, $\partial \Omega_* = \partial_+Y \sqcup \hat \Sigma$, with $\hat \Sigma$ closed, nonempty, and contained in $Y$. 

    Moreover, with respect to the unit normal vector field $\eta$ that points outwards of $\Omega_*$, $\hat \Sigma$ has mean curvature 
    \begin{equation} \label{curvatura media de sigma}
        H = h - a \cdot d( \ln w)(\eta).
    \end{equation} 
    \noindent and $\hat \Sigma$ satisfies the following stability inequality: for every $\varphi \in C^\infty(\hat \Sigma)$,
    \begin{equation} \label{desigualdade de estabilidade crua}
        \begin{split}
            0 \le  \int_{\hat \Sigma} w^a &\Big( 
     |\nabla \varphi|^2  - \varphi^2(|A_\Sigma|^2+\overline{Ric}(\eta,\eta)) - a w^{-2} dw(\eta)^2 \varphi^2 \\
            & \quad + aw^{-1} (\bar \Delta w - \Delta w - H dw(\eta))\varphi^2  - dh(\eta) \varphi^2 \Big)
        \end{split}
    \end{equation}

    \noindent where $\overline \nabla$ denotes the connection of the ambient space $N$, $\nabla$ denotes the connection of $\Sigma$, $\overline{Ric}$ denotes the Ricci curvature of $N$, and $\bar \Delta$ (resp. $\Delta$) denotes the Laplacian on $N$ (resp. $\Sigma$).
\end{teo}

    Using the notation of the above theorem, $\hat \Sigma$ will be called a \textit{$\mu$-bubble} for the parameters $(Y,h,w,a)$.

\begin{proof}
    The existence of $\mu$-bubbles is proved in \cite{generalized-soap-bubbles}, Proposition 12 (see also references therein). The $\mu$-bubbles considered here are minimizers for the functional presented in \cite{Mazet}, Section 4.1. (See also Theorem 4.1 in \cite{Chodosh-Li-Minter-Stryker}.)
\end{proof}

Our first proposition defines some constants that will control parameters in our main theorems of this section. We omit the computations.

    \begin{prop} \label{proposicao que define as constantes do metodo do mazet}
    Let $3 \le k \le 6$ and $\alpha \in (\frac{k-2}{k-1}, \frac{2}{\sqrt{k-1}})$. Consider the real valued function $$m(\alpha,k) :=\min \{\frac{4-(k-1)\alpha^2}{k-\alpha(k-1)}, 4- \frac{k-2}{k-1}\frac{4}{\alpha} \}.$$ There exists $\alpha_*(k) \in (\frac{k-2}{k-1}, \frac{2}{\sqrt{k-1}})$ such that 

    \begin{equation*}
    \begin{cases}
        m(\alpha,k) = 4- \frac{k-2}{k-1}\frac{4}{\alpha} & \text{for }\quad \alpha \in (\frac{k-2}{k-1}, \alpha_*(k)], \\
        m(\alpha,k) = \frac{4-(k-1)\alpha^2}{k-\alpha(k-1)} & \text{for } \quad \alpha \in [\alpha_*(k), \frac{2}{\sqrt{k-1}}).
    \end{cases}  
    \end{equation*}

    The exact values of $\alpha_*(k)$ for each $3 \le k \le 6$ are 
    
    \begin{equation*}
        \begin{cases}
            \alpha_*(3) = 1 \,, \\
            \alpha_*(4) = \frac{2}{3}\left(2- \frac{1}{(\sqrt{2}-1)^{1/3}} + (\sqrt{2}-1)^{1/3} \right), \\
            \alpha_*(5) = \frac{1}{6}\left( 8- \frac{20}{(9\sqrt{201}-91)^{1/3}} + (9\sqrt{201}-91)^{1/3}\right), \\
            \alpha_*(6) = \frac{2}{3} \left( 2 + \frac{(18\sqrt{11}-43)^{1/3}}{5^{2/3}} - \frac{7}{5^{1/3}(18\sqrt{11}-43)^{1/3} }\right).
        \end{cases}
    \end{equation*}
    Moreover, for each $3 \le k \le 6$, the function $\alpha \mapsto m(\alpha,k)$ increases in the interval $(\frac{k-2}{k-1},\alpha_*(k)]$, and decreases in the interval $[\alpha_*(k), \frac{2}{\sqrt{k-1}})$. 
\end{prop}

The approximate values for $\alpha_*$ are $\alpha_*(4)  \sim 0.936$,  $\alpha_*(5)  \sim 0.883$, $\alpha_*(6)  \sim 0.848$. 

We will use the positive real valued function $m=m(\alpha,k)$ described in Proposition \ref{proposicao que define as constantes do metodo do mazet} to state the main theorem of this section. The theorem can be informally described as follows: if a complete Riemannian manifold has uniformly positive 
$\alpha$-bi-Ricci curvature in the spectral sense, then, given a compact domain whose boundary can be written as the union of two disjoint pieces that are sufficiently far apart, there exists a hypersurface separating these pieces whose volume is bounded above by a prescribed constant. The precise statement is given below.

In what follows, we use the following notation, for a Riemannian manifold $N$, 
\begin{equation} \label{def de lambda-Ricci}
    \lambda_{Ric}^N(p) := \inf_{|v|=1,v \in T_p N} Ric^N(v),
\end{equation}
\begin{equation} \label{def de Lambda-bi-Ricci}
    \Lambda_\alpha^N(p) := \inf \Big\{ BRic_\alpha^N (u,v): u,v \text{ are orthonormal in } T_p N \Big\}
\end{equation}

\noindent for any $p \in N$. We omit the superscript when it is clear from the context. 

\begin{teo} \label{mu-bubbles em bi-ricci com peso positivo no sentido espectral}
    Let $(N^{k+1},g)$ be a complete Riemannian manifold of dimension $4 \le k+1 \le 7$. Let $\alpha \in \Big(\frac{k-2}{k-1}, \frac{2}{\sqrt{k-1}}\Big)$, $0<a<m(\alpha,k)$, and $\delta>0$.

    Suppose that $(N,g)$ satisfies 
    \begin{equation*}
        \lambda_1( - a \Delta_g + \Lambda_\alpha) \ge \delta.
    \end{equation*}

    There exist real numbers $L>0$, $V>0$ and $D>0$, that depend only on $(\alpha,a,k,\delta)$, with the following property.
    
    If $X$ is a smooth compact domain of $N$, such that $\partial X = \partial_+X \cup \partial_-X$, where $\partial_-X$ and $\partial_+X$ are nonempty and disjoint, and $X$ satisfies the inequality $d_N(\partial_+X,\partial_-X) \ge L$, then there exists a connected, relatively open subset $\Omega_*$ of $X$ with smooth compact boundary, such that 
    \begin{itemize}
        \item $\Omega_* \supset \partial_-X$ and $\partial \Omega_* = \hat \Sigma \cup \partial_-X$, with $\hat \Sigma$ and $\partial_-X$ disjoint and nonempty.
        \item The hypersurface $\hat \Sigma$ is contained in the interior of $X$.
        \item Every point of $\Omega_*$ is at distance at most $L$ of $\partial_-X$ in $N$.
        \item Every connected component $\Sigma$ of $\hat \Sigma$ satisfies the volume bound $|\Sigma| \le V$. 
        
        If $N$ has dimension four, then $\Sigma$ also satisfies the diameter bound $\text{diam}(\Sigma) \le D$.
    \end{itemize}

\end{teo}

        The proof of Theorem \ref{mu-bubbles em bi-ricci com peso positivo no sentido espectral} follows a strategy introduced by L. Mazet \cite{Mazet}. Several particular cases of this theorem were previously established in the literature: the case $(k,a,\alpha) = (3,1,1)$ was treated by O. Chodosh, C. Li, P. Minter, and D. Stryker \cite{Chodosh-Li-Minter-Stryker}; the case $(k,a,\alpha) = \left(4, \frac{11}{10}, \frac{40}{43}\right)$ by L. Mazet \cite{Mazet}; and the case $(k,a,\alpha) = \left(4,\frac{111}{100},\frac{93}{100}\right)$ by J. Chen, H. Hong, and H. Li \cite{CHL}. Observe that the value of $\alpha$ chosen in \cite{CHL} lies closer to $\alpha_*(4)$ than the one selected in \cite{Mazet}. The choice of these specific parameter values is motivated by the emphasis in those works on stable CMC hypersurfaces in Euclidean space. To treat more general ambient spaces, it is relevant to understand the range of admissible parameters in Theorem \ref{mu-bubbles em bi-ricci com peso positivo no sentido espectral} (see Remark \ref{OBS sobre curvature critica otima em H6}). 
        
    In order to prove Theorem \ref{mu-bubbles em bi-ricci com peso positivo no sentido espectral}, we first establish a lemma and two propositions. The lemma constructs a function $h$ that will serve as a parameter in the construction of $\mu$-bubbles, while the propositions provide curvature estimates for these $\mu$-bubbles. With this preparation, the theorem is then proved by applying the spectral Bishop–Gromov and Bonnet–Myers theorems of G. Antonelli and K. Xu \cite{Antonelli-Xu}.

\begin{lema}[\textit{cf.} \cite{Chodosh-Li-Minter-Stryker}, Theorem 4.1] \label{standard mu-bubble prescribing function h}
    Let $(N,g)$ be a complete Riemannian manifold. Let $X \subset N$ be a smooth compact domain, with boundary $\partial X = \partial_+X \sqcup \partial_-X$, where both $\partial_+X$ and $\partial_-X$ are nonempty hypersurfaces.     Let $\varphi_0:X \to \mathbb{R}$ be a smoothing of the function $d_{\partial_+X}$ that measures the distance in $N$ to $\partial_+X$, such that $\frac{1}{2} d_{\partial_+X} \le \varphi_0 \le 2 d_{\partial_+X}$, and $|\nabla \varphi_0| \le 2$.
    
    Consider a real number $B>0$ and suppose that $d_N(\partial_+X,\partial_-X) \ge \frac{4\pi}{B}+2$.

    Then there exists $\varepsilon \in (0,\frac{1}{2})$ such that
    \begin{equation*}
        Y := \{ \varepsilon < \varphi_0 < \frac{2\pi}{B}+2\varepsilon\}
    \end{equation*}
    is an open subset of $X$ with smooth boundary, $\partial Y = \partial_+ Y \sqcup \partial_- Y$, $\partial_-Y = \{\varphi_0 = \frac{2\pi}{B}+2 \varepsilon\}$, $\partial_+Y = \{\varphi_0 = \varepsilon\}$, and every point of $Y$ is at distance at most $\frac{4\pi}{B}+2$ to $\partial_+X$ in $N$. Moreover, there exists a smooth function $h: Y \to \mathbb{R}$ satisfying
    \begin{equation*}
    \begin{split}
        \begin{cases}
           & h \to +\infty \,\,\,\,\,\, \text{ on } \,\,\, \partial_+Y, \\
           & h \to -\infty \,\,\,\,\,\, \text{ on } \,\,\, \partial_- Y, \\
           &|\nabla h| \le B(1+h^2).
        \end{cases}
    \end{split}        
    \end{equation*}

\end{lema}

\begin{proof}

    Let $\varepsilon \in (0,\frac{1}{2})$ be such that $\varepsilon$ and $\frac{2\pi}{B}+2\varepsilon$ are regular values of $\varphi_0$. Define $\varphi:= \frac{\pi}{2}-\pi \cdot \frac{\varphi_0 -\varepsilon}{ \frac{2\pi}{B}+\varepsilon}$ and note that $\Omega = \{ -\pi/2 < \varphi < \pi/2\}$. We define $h := -\tan(\varphi)$, and compute $$\nabla h = -(1+\tan ^2(\varphi))\nabla \varphi = -(1+h^2)\nabla \varphi$$ so that $|\nabla h| \le B(1+ h^2)$.
 
\end{proof}

The matrix
\begin{equation} \label{matrix G(a,alfa,k)}
    G(a,\alpha,k) =\begin{pmatrix}
        (\frac{1}{k} + \frac{\alpha}{k} - \frac{\alpha}{k^2} + \frac{1}{a} - 1) & \frac{\alpha}{2}(1 -\frac{2}{k}) & \frac{1}{2} - \frac{1}{a} \\
        \frac{\alpha}{2}(1 -\frac{2}{k})  & \frac{k}{k-1} - \alpha & 0 \\
        \frac{1}{2} - \frac{1}{a} & 0 & \frac{1}{a}
    \end{pmatrix}
\end{equation}

\noindent will appear as the matrix of a quadratic form in our calculations. The following proposition, which can be deduced by straightforward computations, provides the algebraic information needed for the proof of Theorem \ref{mu-bubbles em bi-ricci com peso positivo no sentido espectral}. 

\begin{prop} \label{proposicao com condicoes sobre os parametros para aplicar bishop-gromov nas mu-bubbles}
    Suppose that $k \ge 3$ is an integer and $(a,\alpha)$ are real numbers satisfying $a \in (0,4)$ and $\alpha \in (0,2]$. Then the inequalities
    \begin{equation} \label{condições necessarias e suficientes para o método do Mazet funcionar}
    \begin{split}
        \begin{cases}
            & \frac{k-2}{k-1} < \alpha < \frac{2}{\sqrt{k-1}}, \\
            & 0<a \le  4- \frac{k-2}{k-1}\frac{4}{\alpha}, \\
            & a < \frac{4-(k-1)\alpha^2}{k-\alpha(k-1)}
        \end{cases}
    \end{split}
\end{equation}

    \noindent are satisfied if, and only if, 
    \begin{equation} \label{desigualdades desajadas para aplicar bishop gromov espectral}
    \begin{split}
        \begin{cases}
            & \frac{4}{4-a} \frac{1}{\alpha} \le \frac{k-1}{k-2}, \\
            & \text{The matrix }G=G(a,\alpha,k) \text{ given by \eqref{matrix G(a,alfa,k)} is positive definite.}
        \end{cases}
    \end{split}
\end{equation}
    
     Moreover, for integers $k \ge 3$ we have $\frac{k-2}{k-1} < \frac{2}{\sqrt{k-1}} \Leftrightarrow k \le 6$.  
\end{prop}

    Next, we derive a consequence of the stability inequality for $\mu$-bubbles. 

\begin{prop}  \label{segunda variação para mu-bubbles}

Let $N^n$ be a complete Riemannian manifold. Let $a \in (0,4)$ and $\alpha \in (0,2]$ be such that $G = G(a,\alpha,n-1)$ given by \eqref{matrix G(a,alfa,k)} is positive definite. Suppose that there exist $\bar V \in C^\infty(N)$ and a positive function $w \in C^\infty(N)$ satisfying $\bar V \ge \delta - \bar \Lambda_\alpha$ and
    \begin{equation} \label{EDP da função w}
        -a \bar \Delta w = \bar V w.
    \end{equation}

Let $\hat \Sigma$ be a $\mu$-bubble for the parameters $(Y,h,w,a)$, as in Theorem \ref{definicao de mu-bubble}, and $\Sigma$ a connected component of $\hat \Sigma$. Then there exists a real number $\theta = \theta(a,\alpha,n) > 0$ such that 
        \begin{equation} \label{desigualdade desejada para construir mu-bubbles}
        \begin{split}
            \frac{4}{4-a}\int_\Sigma  |\nabla \psi|^2  \ge &\int_\Sigma \Big (\delta - \alpha \lambda_{Ric}^\Sigma + \theta h^2 +dh(\eta)\Big)\psi^2 
        \end{split}
    \end{equation}

    \noindent for every $\psi \in C^\infty(\Sigma)$, where $\eta$ denotes the unit normal vector field to $\Sigma$, pointing outwards of $\Omega_*$. 
\end{prop}

\begin{proof}

  In what follows, $\overline \nabla$ denotes the connection of the ambient space $N$, while $\nabla$ denotes the connection of $\Sigma$, and similar conventions are used for curvature quantities. 

 We use $\varphi = w^{-a/2} \psi$ as a test function in \eqref{desigualdade de estabilidade crua} to use the information \eqref{EDP da função w}. Then
\begin{equation*}
    \begin{split}
        0 \le \int_\Sigma &|\nabla \psi|^2 - aw^{-1}\psi(\nabla w, \nabla \psi) + \frac{a^2}{4} \psi^2 w^{- 2} |\nabla w|^2\\
        & +\psi^2 \Big( 
   - (|A_\Sigma|^2+\overline{Ric}(\eta,\eta)) - a w^{-2} dw(\eta)^2  \\
        &+ aw^{-1} (\bar \Delta w - \Delta w - H dw(\eta))  - dh(\eta)  \Big).
    \end{split}
\end{equation*}

Using \eqref{EDP da função w} we have $ -\frac{\bar \Delta w}{w} = \frac{1}{a} \bar V \ge \frac{1}{a}(\delta - \bar \Lambda_\alpha)$ so that 
\begin{equation*}
    \begin{split}
        \int_\Sigma \Big (\delta - \overline  \Lambda_\alpha + |A_\Sigma|^2+\overline{Ric}(\eta,\eta) \Big)\psi^2  \le \int_\Sigma &|\nabla \psi|^2 - aw^{-1}\psi(\nabla w, \nabla \psi) + \frac{a^2}{4} \psi^2 w^{- 2} |\nabla w|^2\\
        & +\psi^2 \Big( 
    - a w^{-2} dw(\eta)^2 - aw^{-1} (  \Delta w + H dw(\eta))  - dh(\eta)  \Big). 
    \end{split}
\end{equation*}

We integrate by parts the term $- aw^{-1} \Delta w \psi^2$ and collect similar terms to obtain
\begin{equation*}
    \begin{split}
        \int_\Sigma \Big (\delta - \overline  \Lambda_\alpha + |A_\Sigma|^2+\overline{Ric}(\eta,\eta) \Big)\psi^2  \le \int_\Sigma &|\nabla \psi|^2 + aw^{-1}\psi(\nabla w, \nabla \psi) + \frac{a^2-4a}{4} \psi^2 w^{- 2} |\nabla w|^2\\
        & +\psi^2 \Big( 
    - a  [(d \ln w)(\eta)]^2 - a H d (\ln w)(\eta)  - dh(\eta)  \Big).
    \end{split}
\end{equation*}

Now we use Young's inequality
\begin{equation*}
    aw^{-1}\psi(\nabla w, \nabla \psi) \le a \Big( \varepsilon \psi^2 w^{- 2} |\nabla w|^2 + \frac{1}{4 \varepsilon} |\nabla \psi|^2 \Big),
\end{equation*}

\noindent and choose $\varepsilon >0$ so that $a \cdot \varepsilon = - a\frac{a-4}{4} $, which is possible because $0 < a < 4$. Then, denoting the second fundamental form of $\Sigma$ in $N$ by $A_\Sigma$, we have
\begin{equation} \label{conseq da segunda variação da weithed warped mu-bubble com H}
    \begin{split}
        \frac{4}{4-a}\int_\Sigma  |\nabla \psi|^2  \ge &\int_\Sigma \Big (\delta - \overline  \Lambda_\alpha + |A_\Sigma|^2+\overline{Ric}(\eta,\eta) \Big)\psi^2 \\
        &\quad \quad+  \Big( 
     a  [(d \ln w)(\eta)]^2 + a H d (\ln w)(\eta) +dh(\eta)  \Big) \psi^2 
    \end{split}
\end{equation}

To reach the desired inequality \eqref{desigualdade desejada para construir mu-bubbles}, we need to estimate the right-hand side of \eqref{conseq da segunda variação da weithed warped mu-bubble com H} from below. We will see that the integrand in the right-hand side is related to the following quantity 
\begin{equation*}
    \Theta := |A_\Sigma|^2 + \alpha H A_{11} -\alpha \sum_{j=1}^k A_{1j}^2 + a  [(d \ln w)(\eta)]^2 + a H d (\ln w)(\eta)
\end{equation*}

\noindent where $k = n-1$ be the dimension of $\Sigma$, and $A=A_\Sigma$. 

For a fixed, but arbitrary, $p \in \Sigma$, we pick a unit vector $e_1 \in T_p M$ such that $Ric(e_1,e_1) = \lambda_{Ric}^\Sigma(p)$. Using this, and the Gauss equation, we compute at $p$
\begin{equation} \label{estimativa de curvatura no que precede a def de Theta}
    \overline{Ric}(\eta,\eta)  - \overline \Lambda_\alpha \ge \overline{Ric}(\eta,\eta) -\overline {BRic}_\alpha(\eta,e_1) = -\alpha \lambda_{Ric}^\Sigma + \alpha \sum_{j=2}^{k} (A_{ii} A_{11} - A_{1j}^2).
\end{equation}

Note that
\begin{equation} \label{manipulacao com curvatura media que precede definicao de H}
    \sum_{j=2}^{k} (A_{ii} A_{11} - A_{1j}^2) = A_{11}(H-A_{11}) -\sum_{j=2}^k A_{1j}^2 = A_{11}H - \sum_{j=1}^k A_{1j}^2.
\end{equation}

Using equations \eqref{estimativa de curvatura no que precede a def de Theta} and \eqref{manipulacao com curvatura media que precede definicao de H}, we conclude that in order to obtain \eqref{desigualdade desejada para construir mu-bubbles} from \eqref{conseq da segunda variação da weithed warped mu-bubble com H} it is enough to estimate $\Theta$ from below by $\theta \cdot h^2$, for some real number $\theta= \theta(a,\alpha,k)>0$.

From \eqref{curvatura media de sigma} we have $d(\ln w) (\eta) = \frac{h-H}{a}$. Using this, and introducing the traceless part of the second fundamental form of $\Sigma$, $\Phi := A - \frac{H}{k}g$, we rewrite $\Theta$ as 
\begin{equation*}
    \Theta = \frac{1}{k}H^2 + |\Phi|^2 + \alpha H (\Phi_{11} + \frac{H}{k}) - \alpha (\Phi_{11}+ \frac{H}{k})^2 -\alpha \sum_{j>1}^k \Phi_{1j}^2 + \frac{1}{a}(H-h)^2 + H(h-H).
\end{equation*}

Note that
\begin{equation*}
    |\Phi|^2 = \sum_{i,j} \Phi_{ij}^2 = \sum_{i=1}^k \Phi_{ii}^2 + 2 \sum_{j>i} \Phi_{ij}^2 \ge \frac{k}{k-1}\Phi_{11}^2 + 2 \sum_{j>i} \Phi_{ij}^2
\end{equation*}

\noindent since $\Phi$ is symmetric and traceless. Using that $\alpha \le 2$ we obtain $|\Phi|^2 - \alpha \sum_{j>1} |\Phi_{1j}|^2 \ge \frac{k}{k-1}\Phi_{11}^2$. 

Therefore,
\begin{equation} \label{Theta como forma quadratica}
    \Theta \ge  \frac{1}{k}H^2 + \frac{k}{k-1}\Phi_{11}^2 + \frac{\alpha}{k}H^2 + \alpha H \Phi_{11} - \alpha (\frac{1}{k}H + \Phi_{11})^2+  \frac{1}{a}(H-h)^2 + H(h-H).
\end{equation}

The task of estimating $\Theta$ from below has been reduced to the search for a lower bound for the right-hand side of \eqref{Theta como forma quadratica}, which is a quadratic form on $(H,\Phi_{11},h)$ associated to the matrix $G=G(a,\alpha,k)$ given by \eqref{matrix G(a,alfa,k)}. Since $G$ is positive definite, and the set of positive definite matrices is open, we can find $\theta = \theta(a,\alpha,k)$ such that $G - \theta \cdot Id$ is positive definite. Using \eqref{Theta como forma quadratica} we conclude that $\Theta \ge \theta \cdot h^2$. From equations \eqref{estimativa de curvatura no que precede a def de Theta}, \eqref{manipulacao com curvatura media que precede definicao de H}, and \eqref{conseq da segunda variação da weithed warped mu-bubble com H} we obtain the desired inequality \eqref{desigualdade desejada para construir mu-bubbles}.

\end{proof}

Next, we present the proof of Theorem \ref{mu-bubbles em bi-ricci com peso positivo no sentido espectral}. 

\begin{proof}[Proof of Theorem \ref{mu-bubbles em bi-ricci com peso positivo no sentido espectral}]
    
    By Proposition \ref{proposicao com condicoes sobre os parametros para aplicar bishop-gromov nas mu-bubbles} we know that $G = G(a,\alpha,k)$ given by \eqref{matrix G(a,alfa,k)} is positive definite. Let $\theta=\theta(a,\alpha,k)>0$ be as in Proposition \ref{segunda variação para mu-bubbles}. Let $\varphi_0$ be a smoothing of function that measures the distance to $\partial_-X$ in $N$, so that $\varphi_0$ satisfies the hypotheses of Lemma \ref{standard mu-bubble prescribing function h}. Let $B:= \min \{\frac{\delta}{4},\theta\}$ and define $L:= \frac{4\pi}{B}+2$. We use Lemma \ref{standard mu-bubble prescribing function h} with $B$ and $\varphi_0$ as above, to construct $Y$ and the function $h: Y \to \mathbb{R}$. We approximate the continuous function $\Lambda_\alpha$ to obtain $V \in C^\infty(N)$ such that $|| V - \Lambda_\alpha||_{C^0(N)} < \frac{\delta}{4}$. Note that
    \begin{equation*}
        \sup_{0 \neq \phi \in C^\infty_0(N)} \frac{\int_N a_\alpha|\nabla\phi|^2 + V\phi^2}{\int_N \phi^2} \ge \frac{3\delta}{4}. 
    \end{equation*}

    Let $\bar V:= \frac{3\delta}{4} - V$. By Theorem 1 in \cite{fischer-colbrie--schoen}, there exists $w \in C^\infty(N)$ such that $-a \Delta w = \bar V w$, and $w>0$. Associated to $Y$, $a$, $w$ and $h$, is the $\mu$-bubble $\hat \Sigma$ as in Theorem \ref{definicao de mu-bubble}, which is the boundary of $\Omega_{**}$: a relatively open set of the closure of $Y$ with smooth compact boundary, such that $\Omega_{**} \supset \partial_+Y$ and $\partial \Omega_{**} = \hat \Sigma \sqcup \partial_+Y$. Recall from the proof of Lemma \ref{standard mu-bubble prescribing function h} that there exists $\varepsilon \in (0,\frac{1}{2})$ such that $Y := \{ \varepsilon < \varphi_0 < \frac{2\pi}{B}+2\varepsilon\}$; recall also that $\frac{1}{2} d_{\partial_+X} \le \varphi_0 \le 2 d_{\partial_+X}$, and $|\nabla \varphi_0| \le 2$. We add $\{\varphi_0 \le \varepsilon\}$ to $\Omega_{**}$ to form $\Omega_*$. 
    
    Note that every point of $\Omega_*$ is at distance at most $L$ of $\partial_+X$ in $N$. 

    Notice that $\bar V \ge \frac{\delta}{2} - \Lambda_\alpha$. By Proposition \ref{segunda variação para mu-bubbles}, each connected component $\Sigma$ of $\hat \Sigma$ satisfies 
    \begin{equation*}
        \begin{split}
            \frac{4}{4-a}\int_\Sigma  |\nabla \psi|^2  \ge &\int_\Sigma \Big (\frac{\delta}{2} - \alpha \lambda_{Ric}^\Sigma + \theta h^2 +dh(\eta)\Big)\psi^2  \\
            \ge &\int_\Sigma \Big (\frac{\delta}{4} - \alpha \lambda_{Ric}^\Sigma \Big)\psi^2 
        \end{split}
    \end{equation*}

    \noindent for every $\psi \in C^\infty(\Sigma)$, where we used $|\nabla h| \le B(1+h^2)$ in the last inequality. By Proposition \ref{proposicao com condicoes sobre os parametros para aplicar bishop-gromov nas mu-bubbles} we have $\frac{4}{4-a}\frac{1}{\alpha} \le \frac{k-1}{k-2}$, where $k:=n-1$ is the dimension of $\hat \Sigma$, so we can apply the spectral Bishop-Gromov theorem \cite{Antonelli-Xu} to bound the volume of each connected component of $\hat \Sigma$ by a constant $V=V(\alpha,a,k,\delta)$. When $n=4$, the diameter estimate for each component of $\hat \Sigma$ follows from the spectral Bonnet-Myers theorem \cite{Antonelli-Xu}.
\end{proof}

\subsection{Applications}
We conclude this section with applications of Theorem \ref{mu-bubbles em bi-ricci com peso positivo no sentido espectral}. These applications restrict the geometry of Riemannian manifolds with uniformly positive $\alpha$-bi-Ricci curvature in spectral sense. Before stating our main application, we recall some notation. The positive real valued function $m=m(\alpha,k)$ described in Proposition \ref{proposicao que define as constantes do metodo do mazet} will be used as a parameter to control the curvature of the Riemannian manifolds under consideration. Recall also the notation \eqref{def de lambda-Ricci} and \eqref{def de Lambda-bi-Ricci}.
    
\begin{teo} \label{Aplicacao à geometria Riemanniana}
     Let $(N^{k+1},g)$ be a complete Riemannian manifold of dimension $4 \le k+1 \le 7$ and infinite volume. Suppose that $N$ is simply connected and has a finite number $m \in \mathbb{N}$ of ends. 
     
     Let $\alpha \in \Big(\frac{k-2}{k-1}, \frac{2}{\sqrt{k-1}}\Big)$ and $0<a < m(\alpha,k)$.
     
     If $(N,g)$ satisfies  
    \begin{equation*}
        \lambda_1( - a \Delta_g + \Lambda_\alpha) > 0,
    \end{equation*}

    \noindent then there exist a constant $V>0$ and an exhaustion of $N$ by smooth compact domains $\Omega_j$ such that each $\Omega_j$ has exactly $m$ boundary components, and each connected component of the boundary of $\Omega_j$ has its volume bounded from above by $V$.
\end{teo}

 Theorem \ref{Aplicacao à geometria Riemanniana} is related to the study of the macroscopic dimension of Riemannian manifolds with uniformly positive curvature conditions. This is directly related to the heuristic presented in the introduction of the work of O. Chodosh, C. Li, P. Minter, and D. Stryker \cite{Chodosh-Li-Minter-Stryker}.

    \begin{proof}[Proof of Theorem \ref{Aplicacao à geometria Riemanniana}]
        Let $\delta:=\lambda_1( - a \Delta_g + \Lambda_\alpha)>0$. Under the hypotheses of the theorem, we can apply Theorem \ref{mu-bubbles em bi-ricci com peso positivo no sentido espectral} to obtain an exhaustion $\Omega_j$ of $N$ by precompact smooth domains such that each connected component $\Sigma$ of $\Sigma_j := \partial \Omega_j$ satisfies $ |\Sigma| \le V$ for some constant $V = V(\alpha,a,k,\delta)$. Adding to each $\Omega_j$ the bounded components of its complement, we may assume that the complement of $\Omega_j$ consists only of unbounded components. Since $N$ is simply connected, the boundary $\Sigma_j$ of $\Omega_j$ has $m$ connected components.        
    \end{proof}

    The geometric restriction presented in the claim of Theorem \ref{Aplicacao à geometria Riemanniana}, together with the infinite volume assumption, clearly shows that one is unable to bound the volume of regions of $N$ by a universal constant times a power of the area of their boundary, as we note in the right circular cylinder in Euclidean three-space. In particular, the \textit{Cheeger constant} of $N$ is zero.

 To finish this section, we state a more technical result that will be useful to us, which can be obtained by means of a modification in the construction used in the proof of Theorem \ref{Aplicacao à geometria Riemanniana}.  
 
\begin{lema} \label{Aplicacao à geometria Riemanniana v2}
     Let $(N^{k+1},g)$ be a complete Riemannian manifold of dimension $4 \le k+1\le 7$ and infinite volume. Suppose that $N$ is orientable, $H^1_c(N; \mathbb{R})$ is finite dimensional, $N$ has a finite number of ends, and that $N$ has Ricci curvature bounded from below. Let $\alpha \in \Big(\frac{k-2}{k-1}, \frac{2}{\sqrt{k-1}}\Big)$ and $0<a < m(\alpha,k)$.
     
     If there exist a point $p \in N$, and real numbers $\delta>0$ and $R>0$ such that $(N,g)$ satisfies 
    \begin{equation} \label{controle fora de uma bola}
        \int_N a |\nabla \phi|^2 + \Lambda_\alpha \phi^2 \ge \delta \int_N \phi^2
    \end{equation}
    \noindent for every $\phi \in C^\infty_0(N \backslash B_R(p))$, then $\lambda_1(N) = 0$ and $N$ does not satisfy the isoperimetric inequality.
\end{lema}

\begin{proof}
     As in the proof of Theorem \ref{mu-bubbles em bi-ricci com peso positivo no sentido espectral}, the control given by the inequality \eqref{controle fora de uma bola} allows us to construct $w \in C^\infty(N)$ such that: $w>0$ on $N \backslash B_R(p)$, and $-a \Delta w = \bar Vw$ on $N \backslash B_R(p)$, where $\bar V$ satisfies $\bar V \ge \frac{\delta}{2} - \Lambda_\alpha$. Now, let $O_j$ be an exhaustion of $M$ by smooth precompact open sets such that $B_R(p)$ is precompact in $O_1$. The region between $O_{i+1}$ and $O_i$ is contained in $N \backslash B_R(p)$, so we can use Theorem \ref{definicao de mu-bubble} to find a $\mu$-bubble $\hat \Sigma_i$ that separates this region. As in the proof of Theorem \ref{Aplicacao à geometria Riemanniana}, the hypotheses on $(a,\alpha,k)$ allow us to rearrange the stability inequality of $\hat \Sigma_i$ and obtain spectral control over its Ricci curvature, so to apply the spectral Bishop-Gromov theorem \cite{Antonelli-Xu} to bound the volume of each connected component $\Sigma_i$ of $\hat \Sigma_i$ by a constant that depends only on $(a,\alpha,k,\delta)$. We replace $O_i$ by the region that contains $O_i$ and has $\hat \Sigma_i$ as its boundary, to construct an exhaustion $\Omega_j$ of $N$ by precompact smooth domains such that each connected component $\Sigma$ of $\Sigma_j := \partial \Omega_j$ satisfies $ |\Sigma| \le V$ for some constant $V = V(a,\alpha,k,\delta)$. Adding to each $\Omega_j$ the bounded components of its complement, we may assume that the complement of $\Omega_j$ consists only of unbounded components. Since $H^1_c(N; \mathbb{R})$ is finite dimensional and $N$ has a finite number of ends, the number $m_j$ of connected components of the boundary of $\Omega_j$ defines a bounded sequence (\textit{cf.} Lemma 6.1 in \cite{Chodosh-Li-PrimeiraProva-R4}). Finally, we argue by contradiction. If $\lambda_1(N)>0$, we use P. Buser's inequality (\textit{cf}. \cite{Buser}, Theorem 7.1) to prove that the Cheeger constant of $N$ is positive. But this is not compatible with the fact that $N$ has infinite volume and admits the exhaustion previously constructed. For a similar reason, $N$ does not satisfy the isoperimetric inequality.
\end{proof}

	\section{Proof of Theorems \ref{quarto teorema principal k > 0} and \ref{quarto teorema principal k>-1}}\label{section:MainResults}

        In order to apply the results of Section \ref{section:: mu-bubbles and weighted bi-Ricci curvature} to prove theorems \ref{quarto teorema principal k > 0} and \ref{quarto teorema principal k>-1}, it is necessary to estimate the weighted bi-Ricci curvature of an immersed CMC hypersurface in a general Riemannian manifold. In the next proposition, we use the Gauss formula to express it in terms of ambient curvature terms, the mean curvature of the immersion, and its traceless second fundamental form $\Phi := A - \frac{H}n{g}$. We omit the computations, because they are straightforward. 

        \begin{lema} \label{Proposição com a expressão do biRicci da subvariedade CMC}
            Let $M$ be a CMC hypersurface immersed in a Riemannian manifold $(X^{k+1},g)$. Fix $p \in M$. Then, for every orthonormal vectors $e_1,e_2 \in T_pM$

    \begin{align*}
        BRic_\alpha^M(e_1,e_2) &=  \Big( Ric_N(e_1,e_1)-sec_N(e_1\land \nu)  \Big)+ \alpha \Big( Ric_N(e_2,e_2)-sec_N(e_2\land \nu) - sec_N(e_1 \land e_2)\Big) \\
        & \quad\quad +H^2(\frac{1}{k} - \frac{1}{k^2} + \frac{\alpha}{k} - 2\alpha \frac{1}{k^2})  - \Phi_{11}^2 - \alpha \Phi_{22}^2  \\
        &\quad \quad + \Phi_{11}H(1-\frac{2}{k} - \alpha \frac{1}{k}) + \alpha \Phi_{22}H(1 - \frac{3}{k}) \\
        & \quad\quad -\alpha \Phi_{11}\Phi_{22} - \sum_{i=2}^k \Phi_{1i}^2 - \alpha\sum_{i=3}^k\Phi_{2i}^2.
    \end{align*}

     \end{lema}

        \subsection{When the ambient manifold has non-negative sectional curvature} \label{subsection:: When the ambient manifold has non-negative sectional curvature}

        \subsubsection{Proof of Theorem \ref{quarto teorema principal k > 0}} \label{subsubsection:: proof of theorem A}

    Let $(N^k,g)$ be a closed Riemannian manifold with non-negative sectional curvature, and dimension $k \in \{0,\dots,6\}$. Consider the Riemannian product $X := N^k \times \mathbb{R}^{6-k}$ between $N$ and a Euclidean factor. Suppose that $M'$ is a complete, non-minimal, CMC hypersurface with finite index immersed in $X$. We will prove that $M'$ is compact.

     We argue by contradiction. By the \textit{Reduction Lemma} (\textit{cf.} Theorem \ref{argumento de redução em homogeneas}) we know that the existence of a complete, non-compact, finite index CMC hypersurface $M'$ with mean curvature $H>0$ immersed in $X$ would imply the existence of a complete, non-compact, strongly stable CMC hypersurface $M$ with the same mean curvature $H>0$ immersed in $X$. 

     We now study the geometry and topology of $M$. For instance, it is known that $M$ must have infinite volume (\textit{cf.} Theorem \ref{Katia Frensel volume infinito}). Moreover, the reduction to the strongly stable case allows for topological simplifications. Using Theorem 1 in \cite{fischer-colbrie--schoen}, we may assume that $M$ is simply connected by passing to its universal cover. Moreover, Theorem 0.1 in \cite{Detang-XuCheng-Cheung} guarantees that $M$ has one end.

    In order to use the theory of $\mu$-bubbles to further restrict the geometry of $M$, we estimate the $\alpha$-bi-Ricci curvature of $M$ in spectral sense in the next proposition.

		\begin{prop} \label{controle espectral do alfa-bi-Ricci em sec nao negativa}

     Let $(X^{6},g)$ be a Riemannian manifold with non-negative sectional curvature. If $M$ is a strongly stable CMC hypersurface with mean curvature $H> 0$ immersed in $X$, then for $a= \frac{11}{10}$ and $\alpha = \frac{40}{43}$, there exists $\delta >0$ such that
    \begin{equation*}
        \lambda_1 ( - a \Delta_M +  \Lambda_{\alpha}^M) \ge \delta. 
    \end{equation*}
		\end{prop}

\begin{proof}[Proof of Proposition \ref{controle espectral do alfa-bi-Ricci em sec nao negativa}]
    Using the strong stability inequality, we have
    \begin{equation*}
        \int_M a|\nabla \psi|^2 + \Lambda_\alpha^M \psi^2 \ge \int_M \big( a (Ric(\nu) +|A|^2)+\Lambda_\alpha^M \big)\psi^2   \label{desig1 da prova do main-theorem com curvatura nao negativa}
    \end{equation*}
    \noindent for every $\psi \in C^{\infty}_0(M)$. Therefore, it is enough to prove
    \begin{equation}
        a (Ric(\nu) +|A|^2)+\Lambda_\alpha^M\ge \delta \label{desig2 da prova do main-theorem com curvatura nao negativa}
    \end{equation}
    \noindent for some $\delta>0$.
    
    In what follows, we do the calculation for $M^k$ in $X^{k+1}$ and choose $k=5$ in the end. 
    
    Notice that    
    \begin{equation}
        |A|^2 = |\Phi|^2 + \frac{H^2}{k} = \sum_{i=1}^k \Phi_{ii}^2 + 2 \sum_{j>i}\Phi_{i,j}^2 + \frac{H^2}{k}.  \label{desig3 da prova do main-theorem com curvatura nao negativa}
    \end{equation}

    Moreover,
    \begin{equation}
        a \cdot 2 \sum_{j>i}\Phi_{i,j}^2 \ge \sum_{i=2}^k \Phi_{1i}^2 + \alpha\sum_{i=3}^k\Phi_{2i}^2 \,,\label{desig4 da prova do main-theorem com curvatura nao negativa}
    \end{equation}

    \noindent because $ a= \frac{11}{10}$ and $\alpha = \frac{40}{43}$ guarantees $2a \ge 1 + \alpha$. 

    We also have     
    \begin{equation} \label{desig5 da prova do main-theorem com curvatura nao negativa}
    \begin{split}
        \sum_{i=1}^k \Phi_{ii}^2 &\ge \Phi_{11}^2 + \Phi_{22}^2 + \frac{1}{k-2}(\Phi_{11}+\Phi_{22})^2 = \frac{k-1}{k-2}\Phi_{11}^2 + \frac{k-1}{k-2}\Phi_{22}^2 + 2\frac{1}{k-2} \Phi_{11}\Phi_{22},
    \end{split}
    \end{equation}

    \noindent where we used the Cauchy-Schwarz inequality and the fact that $\Phi$ is traceless.

    Therefore, we obtain from Lemma \ref{Proposição com a expressão do biRicci da subvariedade CMC} the following estimate, using \eqref{desig3 da prova do main-theorem com curvatura nao negativa}, \eqref{desig4 da prova do main-theorem com curvatura nao negativa}, \eqref{desig5 da prova do main-theorem com curvatura nao negativa}, and the assumption that $X$ has non-negative sectional curvature:    
    \begin{equation*}
    \begin{split}
        a|A|^2+ BRic_\alpha^M(e_1,e_2) &\ge H^2(a \frac{1}{k}+\frac{1}{k} - \frac{1}{k^2} + \frac{\alpha}{k} - \alpha \frac{1}{k^2} - \alpha \frac{1}{k^2}) \\
        & \quad \quad + ( a \frac{k-1}{k-2}-1) \Phi_{11}^2 + (a \frac{k-1}{k-2} - \alpha) \Phi_{22}^2 \\
        & \quad\quad+ \Phi_{11}H(1-\frac{2}{k} - \alpha \frac{1}{k}) + \alpha \Phi_{22}H(1 - \frac{3}{k}) \\
        & \quad\quad+ \Phi_{11}\Phi_{22}(2 a \frac{1}{k-2}-\alpha).
    \end{split}
    \end{equation*}
    
    The right hand side of this inequality is a quadratic form $T$ on $(H,\Phi_{11},\Phi_{22})$. When $k=5$, $a= \frac{11}{10}$ and $\alpha = \frac{40}{43}$, this quadratic form is positive definite. Since the set of positive definite matrices is open, we may find $\theta>0$ such that $T \ge \theta Id$. Hence, at each point $p \in M$, if we pick $e_1,e_2 \in T_pM$ orthonormal such that $\Lambda_{ \alpha}(p) = BRic^M_{ \alpha}(e_1,e_2)$, the above calculations show that
    \begin{equation*}
        a |A|^2 + \Lambda_{\alpha} \ge \theta H^2 \,\,\,\, \text{ at} \,\,\,p.
    \end{equation*}

    Since $p \in M$ is arbitrary and $H>0$, we have proved inequality \eqref{desig2 da prova do main-theorem com curvatura nao negativa} for $\delta = \theta H^2>0$.
\end{proof}

Using Proposition \ref{controle espectral do alfa-bi-Ricci em sec nao negativa} and the topological properties deduced for $M$ we verify the hypotheses of Theorem \ref{Aplicacao à geometria Riemanniana} and conclude that $M$ admits an exhaustion by open precompact smooth sets $\Omega_j$ such that each $\Omega_j$ has exactly one boundary component $\Sigma_j$ and each $\Sigma_j$ has volume bounded from above by a constant $V$ that depends on $M$, but not on $j$.

    This geometric property of $M$ and the fact that it has infinite volume clearly guarantee that its Cheeger constant is zero. But this contradicts P. Buser's inequality (\cite{Buser}, Theorem 7.1), because by the result of Remark \ref{remark sobre estimativa de curvatura} we know that $M$ has bounded second fundamental form, hence bounded curvature, and the strong stability inequality of $M$ readily implies that $\lambda_1(M)>0$.

\qed

    \subsubsection{Additional results}  \label{subsubsection:: additional results}
        In this subsection, we present further partial results obtained using similar techniques. The following theorem applies to more general ambient spaces, under the additional assumption that the hypersurfaces have finite topology.
        
		\begin{teo} \label{teorema com hipotese topologia finita em curvatura nao negativa}
        Let $X$ be an orientable, complete Riemannian manifold of bounded geometry, non-negative sectional curvature and dimension six. Let $M$ be a complete, non-minimal CMC hypersurface immersed in $X$.
        
        \begin{enumerate}
            \item If $M$ is weakly stable and $H^1_c(M;\mathbb{R})$ is finite dimensional, then $M$ is compact.
            \item If $M$ has finite index, $H^1_c(M;\mathbb{R})$ is finite dimensional, and $M$ has a finite number of ends, then $M$ is compact.
        \end{enumerate}
		\end{teo}      

        \begin{proof}
            First, we prove the item $(2)$. On the contrary, suppose that there exists a complete, non-compact, non-minimal CMC hypersurface $M$ immersed in $X$, with a finite number of ends and dim$H^1_c(M;\mathbb{R}) < \infty$, which has finite index. 
            
            From Theorem \ref{estimativa de curvatura mais geral possivel} and the Gauss equation, we derive that $M$ has Ricci curvature bounded from below. Moreover, $M$ is orientable, because $X$ is orientable and $H>0$, and $M$ has infinite volume, because $X$ has bounded geometry (see Theorem \ref{Katia Frensel volume infinito}). 

            Notice that Theorem \ref{CMC com H>Hlambda tem lambda1>0} guarantees that $\lambda_1(M)>0$.
            
            Fix $p \in M$. Since $M$ has finite index, we can use Proposition $1$ in \cite{Fischer-Colbrie} to find $R>0$ such that $M \backslash B_R(p)$ is strongly stable. Using Proposition \ref{controle espectral do alfa-bi-Ricci em sec nao negativa} we verify the hypotheses of Lemma \ref{Aplicacao à geometria Riemanniana v2} to conclude
            that $\lambda_1(M)$ vanishes. But this is a contradiction.

            Finally, we prove the item (1). It is enough to check that under the hypotheses of the theorem, the hypersurface $M$ must have a finite number of ends. But it follows immediately from Theorem 0.1 of \cite{Detang-XuCheng-Cheung} that $M$ has only one end. 
        \end{proof}

    We now turn to ambient manifolds with non-negative sectional curvature and Euclidean volume growth. Sobolev inequalities on such manifolds were studied by S. Brendle \cite{BrendleIsop-curvatura-naonegativa}, and J. Chen, H. Hong, and H. Li \cite{CHL} established results on CMC hypersurfaces in this setting, building on Brendle’s work. We will make use of the results in \cite{CHL} in the arguments that follow.

		\begin{teo} \label{teorema sobre as de crescimento euclideano de volume}
                Let $(X,g)$ be an orientable, six dimensional complete Riemannian manifold of bounded geometry, non-negative sectional curvature and Euclidean volume growth. Let $M$ be a complete, non-minimal CMC hypersurface immersed in $X$. If $M$ has finite index, then $M$ is compact. 
		\end{teo}      

        \begin{proof}
            On the contrary, suppose that there exists a complete, non-compact CMC hypersurface $M$ immersed in $X$ with finite index and mean curvature $H>0$. We first restrict the topology of $M$. Theorem 2.5 in \cite{CHL} guarantees that $\text{dim} H^1(L^2(M)) < \infty$ and $M$ has a finite number of ends. By Proposition 2.3 in \cite{CHL} we know that $M$ satisfies a Sobolev inequality which allows us to use Proposition $2.11$ in \cite{Carron-L2-Harmonic-forms} to show that $\text{dim} H^1_c(M;\mathbb{R}) \le \text{dim} H^1(L^2(M)) < \infty$. Moreover, $M$ is clearly orientable. 
            
            We now restrict the geometry of $M$. Since $X$ has bounded geometry, $M$ has infinite volume (\textit{cf.} Theorem \ref{Katia Frensel volume infinito}). Moreover, by the result of Remark \ref{remark sobre estimativa de curvatura} we know that $M$ has bounded second fundamental form, and hence bounded curvature. Fix $p \in M$. Since $M$ has finite index, there exists $R>0$ such that $M \backslash B_R(p)$ is strongly stable (\textit{cf.} Proposition $1$ in \cite{Fischer-Colbrie}). 
            
            Therefore, using these geometric and topological properties of $M$ and Proposition \ref{controle espectral do alfa-bi-Ricci em sec nao negativa}, we verify the hypotheses of Lemma \ref{Aplicacao à geometria Riemanniana v2} to conclude that $M$ does not satisfy the isoperimetric inequality. But this contradicts Proposition \ref{Garante isoperimetrica}.             
        \end{proof}

    \subsection{When the ambient manifold has a lower bound on the sectional curvature} \label{subsection:: When the ambient manifold has a lower bound on the sectional curvature}

    \subsubsection{Proof of Theorem \ref{quarto teorema principal k>-1}} \label{subsubsection:: proof of theorem B}

Let $X$ be a six-dimensional Riemannian manifold with sectional curvature bounded from below by $-1$ such that the isometry group of $X$ acts cocompactly on $X$. Suppose that $M'$ is a complete, finite index CMC hypersurface immersed in $X$ with mean curvature $H > 7$. We will prove that $M'$ is compact.
            
        The proof is by contradiction. By the \textit{Reduction Lemma} (see Theorem \ref{argumento de redução em homogeneas}) we know that the existence of a complete, non-compact, finite index CMC hypersurface $M'$ with mean curvature $H>7$ immersed in $X$ would imply the existence of a complete, non-compact, strongly stable CMC hypersurface $M$ with the same mean curvature $H>7$ immersed in $X$. 
        
        To complete the proof, we argue in a similar way to what we did in the proof of Theorem \ref{quarto teorema principal k > 0}, but there are two main differences which we highlight. First, the control over the number of ends of $M$ now comes from Theorem \ref{unicidade de fim quando curvature é limitada inferiormente em X6}. Second, the spectral estimate on the weighted bi-Ricci curvature of $M$ is made through Proposition \ref{controle do bi-Ricci com peso da CMC estavel em ambiente de curvatura negativa} instead of Proposition \ref{controle espectral do alfa-bi-Ricci em sec nao negativa}.

        \begin{prop}\label{controle do bi-Ricci com peso da CMC estavel em ambiente de curvatura negativa}
       Let $(X^6,g)$ be a Riemannian manifold with sectional curvature bounded from below $sec_X \ge -1$. Let $M$ be a strongly stable CMC hypersurface with mean curvature $H > 7$ immersed in $X$. Then for $\alpha = \frac{40}{43}$ and $a = \frac{11}{10}$ there exists $\delta >0$ such that $M$ satisfies
    \begin{equation*}
        \lambda_1 ( - a \Delta_g +  \Lambda_{\alpha}) \ge \delta.
    \end{equation*}
		\end{prop}                               

\begin{proof}[Proof of Proposition \ref{controle do bi-Ricci com peso da CMC estavel em ambiente de curvatura negativa}]
     Since $M$ is strongly stable, we have
    \begin{equation}
        \int_M a|\nabla \psi|^2 + \Lambda_\alpha \psi^2 \ge \int_M \big( a (Ric(\nu) +|A|^2)+\Lambda_\alpha \big)\psi^2 \label{desig1 na prova do Teo controle do bi-Ricci com peso da CMC estavel em ambiente de curvatura negativa}
    \end{equation}

    \noindent for every $\psi \in C^{\infty}_0(M)$. 

    Therefore, it is enough to prove
    \begin{equation}
        a  (Ric(\nu)+ |A|^2) +\Lambda_\alpha \ge \delta \label{desig2 na prova do Teo controle do bi-Ricci com peso da CMC estavel em ambiente de curvatura negativa}
    \end{equation}

    \noindent for some $\delta>0$. In what follows, we denote by $k$ the dimension of $M$ and substitute $k=5$ at the end.

    Notice that    
    \begin{equation}
        |A|^2 = |\Phi|^2 + \frac{H^2}{k} = \sum_{i=1}^k \Phi_{ii}^2 + 2 \sum_{j>i}\Phi_{i,j}^2 + \frac{H^2}{k}. \label{desig3 na prova do Teo controle do bi-Ricci com peso da CMC estavel em ambiente de curvatura negativa}
    \end{equation}

    Moreover,
    \begin{equation}
        2a \sum_{j>i}\Phi_{ij}^2 \ge \sum_{i=2}^k \Phi_{1i}^2 + \alpha\sum_{i=3}^k\Phi_{2i}^2 \label{desig4 na prova do Teo controle do bi-Ricci com peso da CMC estavel em ambiente de curvatura negativa}
    \end{equation}

    \noindent because $2a \ge 1+ \alpha$ when $a=\frac{11}{10}$ and $\alpha = \frac{40}{43}$.

    We also have     
    \begin{equation}
    \begin{split}
        \sum_{i=1}^k \Phi_{ii}^2 &\ge \Phi_{11}^2 + \Phi_{22}^2 + \frac{1}{k-2}(\Phi_{11}+\Phi_{22})^2 =\frac{k-1}{k-2}\Phi_{11}^2 + \frac{k-1}{k-2}\Phi_{22}^2 + 2\frac{1}{k-2} \Phi_{11}\Phi_{22} \label{desig5 na prova do Teo controle do bi-Ricci com peso da CMC estavel em ambiente de curvatura negativa}
    \end{split}
    \end{equation}

    \noindent where we used the Cauchy-Schwarz inequality and that $\Phi$ is traceless.

    Fix $p \in M$. Let $e_1,e_2$ be arbitrary orthonormal vectors in $T_pM$. Using Lemma \ref{Proposição com a expressão do biRicci da subvariedade CMC}, equations \eqref{desig3 na prova do Teo controle do bi-Ricci com peso da CMC estavel em ambiente de curvatura negativa}, \eqref{desig4 na prova do Teo controle do bi-Ricci com peso da CMC estavel em ambiente de curvatura negativa} and \eqref{desig5 na prova do Teo controle do bi-Ricci com peso da CMC estavel em ambiente de curvatura negativa}, and the fact that $sec_X \ge -1$ we arrive at
    \begin{equation*}
    \begin{split}
          a  (Ric(\nu)+ |A|^2) + BRic_\alpha^M(e_1,e_2) &\ge - k a - (k-1)(1+\alpha) + \alpha  \\
          &\quad \quad+H^2(a \frac{1}{k}+\frac{1}{k} - \frac{1}{k^2} + \frac{\alpha}{k} - \alpha \frac{1}{k^2} - \alpha \frac{1}{k^2})  \\
        & \quad \quad+ ( a \frac{k-1}{k-2}-1) \Phi_{11}^2 + (a \frac{k-1}{k-2} - \alpha) \Phi_{22}^2 \\
        &\quad \quad+ \Phi_{11}H(1-\frac{2}{k} - \alpha \frac{1}{k}) + \alpha \Phi_{22}H(1 - \frac{3}{k}) \\
        &\quad \quad+ \Phi_{11}\Phi_{22}(2 a \frac{1}{k-2}-\alpha) .
    \end{split}
    \end{equation*}

    Now we write $H = k+\varepsilon$ for some $\varepsilon>2$, to find
    \begin{equation*}
    \begin{split}
          a  (Ric(\nu)+ |A|^2) + BRic_\alpha^M(e_1,e_2) &\ge 2\varepsilon\Big(  a + 1 + \alpha + \frac{1}{k}(-1-2\alpha) \Big) \\
        & \quad\quad+ \varepsilon^2 \Big(  \frac{1}{k}(a + 1 + \alpha) + \frac{1}{k^2}(-1-2\alpha) \Big) \\
        & \quad\quad+ ( a \frac{k-1}{k-2}-1) \Phi_{11}^2 + (a \frac{k-1}{k-2} - \alpha) \Phi_{22}^2 \\
        &\quad\quad+ \Phi_{11}\varepsilon(1-\frac{2}{k} - \alpha \frac{1}{k}) + \alpha \Phi_{22}\varepsilon(1 - \frac{3}{k}) \\
        &\quad\quad+ \Phi_{11}\Phi_{22}(2 a \frac{1}{k-2}-\alpha)  \\
        & \quad\quad+ k \Phi_{11}(1-\frac{2}{k} - \alpha \frac{1}{k}) + \alpha k \Phi_{22}(1 - \frac{3}{k}).
    \end{split}
    \end{equation*}

  In what follows, we substitute $k=5$, $a = \frac{11}{10}$ and $\alpha = \frac{40}{43}$. We remark that these choices of parameters turns into positive definite the quadratic part on $(\varepsilon,\Phi_{11},\Phi_{22})$ of the above expression. Then we obtain
    \begin{equation} \label{desig6 na prova do Teo controle do bi-Ricci com peso da CMC estavel em ambiente de curvatura negativa}
    \begin{split}
          a  (Ric(\nu)+ |A|^2) + BRic_\alpha^M(e_1,e_2) &\ge \frac{1057}{215}\varepsilon +  \frac{1057}{2150}\varepsilon^2  +  \frac{7}{15} \Phi_{11}^2 + \frac{346}{645} \Phi_{22}^2  \\
        &\quad\quad+ \frac{89}{215}\Phi_{11}\varepsilon + \frac{16}{43} \Phi_{22}\varepsilon   - \frac{127}{645}\Phi_{11}\Phi_{22}\\
        & \quad\quad+ \frac{89}{43}\Phi_{11}  + \frac{80}{43}\Phi_{22} \,.
    \end{split}
    \end{equation}

    We introduce a new parameter $\beta>0$ to control the linear terms on entries of $\Phi$ through Young's inequality:
    \begin{equation} \label{desig7 na prova do Teo controle do bi-Ricci com peso da CMC estavel em ambiente de curvatura negativa}
        \frac{89}{43}\Phi_{11}  +\frac{80}{43} \Phi_{22}  \ge - \left( \beta \Phi_{11}^2 + (\frac{89}{43})^2 \frac{1}{4\beta} \right) - \left(\beta \Phi_{22}^2 + (\frac{80}{43})^2\frac{1}{4\beta}\right).
    \end{equation}

    Combining inequalities \eqref{desig6 na prova do Teo controle do bi-Ricci com peso da CMC estavel em ambiente de curvatura negativa} and \eqref{desig7 na prova do Teo controle do bi-Ricci com peso da CMC estavel em ambiente de curvatura negativa} yields
    \begin{equation} \label{desig8 na prova do Teo controle do bi-Ricci com peso da CMC estavel em ambiente de curvatura negativa}
    \begin{split}
          a  (Ric(\nu)+ |A|^2) + BRic_\alpha^M(e_1,e_2) &\ge \frac{1057 }{215}\varepsilon - (\frac{89}{43})^2 \frac{1}{4\beta} - (\frac{80}{43})^2\frac{1}{4\beta} \\
          &\quad\quad+ \frac{1057}{2150}\varepsilon^2  + (\frac{7}{15}-\beta) \Phi_{11}^2 + (\frac{346}{645} - \beta) \Phi_{22}^2 \\
        &\quad\quad+ \frac{89}{215}\Phi_{11}\varepsilon + \frac{16}{43}\Phi_{22}\varepsilon  - \frac{127}{645} \Phi_{11}\Phi_{22}.
    \end{split}
    \end{equation}

    We choose $\beta = \frac{1}{5}$ so that the quadratic form on $(\varepsilon,\Phi_{11},\Phi_{22})$ associated to the matrix
    \begin{equation*}
        Q= \begin{pmatrix}
             \frac{1057}{2150} &  \frac{1}{2}\frac{89}{215} & \frac{1}{2}\frac{16}{43} \\[3pt]
            \frac{1}{2} \frac{89}{215} & (\frac{7}{15}-\beta) &  - \frac{1}{2}\frac{127}{645} \\[3pt]
             \frac{1}{2}\frac{16}{43} & - \frac{1}{2}\frac{127}{645} & (\frac{346}{645} - \beta) 
        \end{pmatrix}
    \end{equation*}

    \noindent is positive definite. Eliminating the quadratic expression from \eqref{desig8 na prova do Teo controle do bi-Ricci com peso da CMC estavel em ambiente de curvatura negativa} and using that $\beta = \frac{1}{5}$ and $\varepsilon>2$ we conclude that
    \begin{equation} \label{desig9 na prova do Teo controle do bi-Ricci com peso da CMC estavel em ambiente de curvatura negativa}
          a  (Ric(\nu)+ |A|^2) + BRic_\alpha^M(e_1,e_2) \ge  \frac{5583}{36980}.
    \end{equation}

    Choosing $e_1,e_2$ so that $BRic_\alpha^M(e_1,e_2) = \Lambda_\alpha(p)$, and using the fact that $p \in M$ is arbitrary, we have verified the desired inequality \eqref{desig2 na prova do Teo controle do bi-Ricci com peso da CMC estavel em ambiente de curvatura negativa}.  
\end{proof}       
        
     With these tools, we complete the proof of Theorem \ref{quarto teorema principal k>-1} following the proof of Theorem \ref{quarto teorema principal k > 0}.

\qed

    \begin{obs}   \label{OBS sobre curvature critica otima em H6}

    We did not try to optimize our choice of lower bound for $|H|$ in the hypothesis of Proposition \ref{controle do bi-Ricci com peso da CMC estavel em ambiente de curvatura negativa}. Moreover, we chose $a = \frac{11}{10}$ and $\alpha = \frac{40}{43}$ in Propositions \ref{controle espectral do alfa-bi-Ricci em sec nao negativa} and \ref{controle do bi-Ricci com peso da CMC estavel em ambiente de curvatura negativa} to facilitate comparison with the work of L. Mazet \cite{Mazet}. Nevertheless, we believe that the flexibility in the parameters $(a,\alpha)$ described in Theorem \ref{mu-bubbles em bi-ricci com peso positivo no sentido espectral} is essential for improving and potentially optimizing the hypothesis on the lower bound on $|H|$ in Proposition \ref{controle do bi-Ricci com peso da CMC estavel em ambiente de curvatura negativa} within our framework.
    \end{obs}

    \subsubsection{Further results}
        Next, we turn our attention to the study of finite index CMC hypersurfaces immersed in more general ambient spaces. 

		\begin{teo} \label{CMC de indice finito em ambiente com curvatura limitada por baixo}
              Let $X$ be an orientable, complete six-dimensional Riemannian manifold with bounded geometry and sectional curvature bounded from below, $sec_X \ge -1$. Let $M$ be a complete CMC hypersurface with mean curvature $|H|>7$ immersed in $X$. 
              \begin{enumerate}
                  \item If $M$ is weakly stable and $H^1_c(M;\mathbb{R})$ is finite dimensional, then $M$ is compact.
                  \item If $M$ has finite index, $H^1_c(M;\mathbb{R})$ is finite dimensional, and $M$ has a finite number of ends, then $M$ is compact.
              \end{enumerate}
		\end{teo}           

        \begin{proof}
            We first prove item (2). On the contrary, suppose that there exists a non-compact, complete, finite index $H$-CMC hypersurface with $H>7$ immersed in $X$. Since $X$ is orientable and $H > 7$, we obtain that $M$ is orientable. By Theorem \ref{Katia Frensel volume infinito}, $M$ has infinite volume. Moreover, we can apply Theorem \ref{CMC com H>Hlambda tem lambda1>0} to conclude that
           $\lambda_1(M) > 0$.

            By Theorem \ref{estimativa de curvatura mais geral possivel} and the Gauss equation, using $sec_X \ge -1$, we find that $M$ has Ricci curvature bounded from below. Finally, we use Proposition \ref{controle do bi-Ricci com peso da CMC estavel em ambiente de curvatura negativa} to verify the hypotheses of Lemma \ref{Aplicacao à geometria Riemanniana v2} and conclude that $\lambda_1(M)$ vanishes. This is a contradiction.  
            
            Now we prove item (1). In light of item (2), it is enough to estimate the number of ends of the hypersurface $M$. But Theorem \ref{unicidade de fim quando curvature é limitada inferiormente em X6} guarantees that $M$ has at most one end. 
        \end{proof}

    \section{Manifolds with uniformly positive curvature conditions} \label{section::Manifolds with uniformly positive curvature conditions}

    This section is devoted to the study of the compactness of complete, two-sided, finite index minimal hypersurfaces immersed in a six-dimensional Riemannian manifold $X$ which satisfies a uniformly positive curvature condition. The main goal of this section is to prove

    \begin{teo}  \label{corolario da extensao do teorema do Catino}
        Let $X$ be a complete six-dimensional Riemannian manifold. Suppose that there exists $\alpha \in (1,\frac{5}{4}]$ so that $BRic_\alpha^X$ is uniformly positive. If $M$ is a complete, two-sided, finite index minimal hypersurface immersed in $X$, then $M$ is compact.
    \end{teo}

        Compare the statement of Theorem \ref{corolario da extensao do teorema do Catino} with the statement of Proposition 2.4 in \cite{Qintao}. In order to prove Theorem \ref{corolario da extensao do teorema do Catino}, we first adapt a sharp result by K. Xu (\cite{KXuDimensionalConstraints}, Theorem 1.10), which applies for Riemannian manifolds with uniformly positive Ricci curvature in spectral sense. When dealing with finite index CMC hypersurfaces, the spectral control over the Ricci curvature is naturally verified in the complement of a compact domain. This motivated the following adaptation of the aforementioned work.

    \begin{teo}[\textit{cf.} \cite{KXuDimensionalConstraints}, Theorem 1.10] \label{adaptacao do K. Xu dimensional constraints}
        Let $(M^n,g)$ be a complete and non-compact Riemannian manifold with dimension $3 \le n \le 7$. Suppose that $0 \le \gamma < \frac{4}{n-1}$ when $n>3$, or $0 \le \gamma \le 2$ when $n=3$. Then
        \begin{equation*}
            \sup_{p \in M, \,R>0} \,\sup_{\phi \in C^{\infty}_0(M\backslash B_R(p))}  \frac{\int_M \gamma |\nabla\phi|^2 + \lambda_{Ric}\phi^2 }{\int_M \phi^2 } \le 0.
        \end{equation*}
    \end{teo}

    \begin{proof}
        We argue by contradiction. Suppose that there exist $p \in M$ and $R>0$ such that $$\sup_{\phi \in C^{\infty}_0(M\backslash B_R(p))}\frac{\int_M \gamma |\nabla\phi|^2 + \lambda_{Ric}\phi^2 }{\int_M \phi^2 }   =: 3 \lambda>0.$$ 

        We approximate the continuous function $\lambda_{Ric}$ to get $V \in C^{\infty}(M)$ such that $||V - \lambda_{Ric}||_{C^0(M)} < \lambda$. Note that 
        \begin{equation*}
            \sup_{\phi \in C^{\infty}_0(M\backslash B_R(p))} \frac{\int_M \gamma |\nabla\phi|^2 + V\phi^2 }{\int_M \phi^2 }   \ge  2\lambda >0.
        \end{equation*}
        Let $\bar V := 2\lambda - V$ and note that $\bar V \ge \lambda - \lambda_{Ric}$. Using Theorem 1 in \cite{fischer-colbrie--schoen} we can find a smooth function $w>0$ such that $-\gamma \Delta w = \bar V w $ on $M \backslash B_R(p)$.

        We use Lemma \ref{standard mu-bubble prescribing function h} with 
        \begin{equation*}
            B <\min \{\left[ \frac{1}{n-1} - \frac{1}{4}(\frac{n-3}{n-1})^2 \frac{1}{\frac{1}{n-1}+\beta-1}\right], \xi \beta^{-1 }\}
        \end{equation*}
        
        \noindent where $\beta := \frac{1}{\gamma}$ and $\xi := \beta \lambda > 0$, to construct $Y \subset \subset M \backslash B_R(p)$, and the corresponding smooth prescribing function $h: Y \to \mathbb{R}$ satisfying $|\nabla h| \le B(1+h^2)$. Let $a := \gamma$. Associated to $Y$, $a$, $w$ and $h$, there is a $\mu$-bubble $\hat \Sigma$, as in Theorem \ref{definicao de mu-bubble}. We compute $\Delta w =  -\beta \bar V w  \le -\beta(\lambda-\lambda_{Ric})w = \beta \lambda_{Ric}w - \beta \lambda w$. Hence $\Delta w \le \beta \lambda_{Ric}w - \xi w$.
        
         We consider $v := w^a$ and compute 
         \begin{equation*}
             \Delta v \le \lambda_{Ric} v - \xi\beta^{-1}v + (1-\beta)v^{-1}|\nabla v|^2.
         \end{equation*}

        Plugging $\varphi = v^{-1}$ in the stability inequality, and using the above inequality, one can prove, as in \cite{KXuDimensionalConstraints}, that under the hypotheses about $\gamma$ of this theorem the stability inequality of $\Sigma$ produces a contradiction.
    \end{proof}

    \begin{proof}[Proof of Theorem \ref{corolario da extensao do teorema do Catino}]

Let $\gamma := \frac{1}{\alpha}$, so that $\frac{4}{5} \le \gamma < 1$. Suppose that $BRic_\alpha^X\ge c > 0$. 
     In order to reach a contradiction, suppose that $M$ is non-compact. The finite index hypothesis on $M$ guarantees that there exists $R>0$ such that
    \begin{equation*}
        \int_M |\nabla f|^2 \ge \int_M (Ric(\nu)+ |A|^2)f^2   
    \end{equation*}

    \noindent for every $f \in C^{\infty}_0(M \backslash B_R)$, thus
    \begin{equation}
        \int_M \gamma|\nabla f|^2+\lambda_{Ric}^Mf^2 \ge \int_M [\gamma(Ric(\nu)+ |A|^2)+\lambda_{Ric}^M]f^2.   \label{desigualdade que vem da hipotese de indice finito v1}
    \end{equation}

    Fix $p \in M$ and pick $e_1 \in T_p M$ so that $\lambda_{Ric}^M(p) = Ric_M(e_1)$. Using an orthonormal basis $e_1,\dots,e_5$, we compute at $p$:
    \begin{equation*}
            \gamma(Ric(\nu)+ |A|^2)+\lambda_{Ric}^M = \gamma Ric(\nu)+ \sum_{i=2}^5 sec_X(e_1 \land e_i) - |Ae_1|^2 + \gamma |A|^2.
    \end{equation*}

    Note that 
    \begin{equation*}
    \begin{split}
        \gamma Ric(\nu)+ \sum_{i=2}^5 sec_X(e_1 \land e_i) &= \gamma \Big( Ric(\nu) + \frac{1}{\gamma}\sum_{i=2}^5 sec_X(e_1 \land e_i) \Big) = \gamma  BRic_\alpha(\nu,e_1) \ge \gamma c > 0,
    \end{split}
    \end{equation*}

    \noindent because $ BRic_\alpha^X(\nu,e_1) \ge c > 0$. Moreover    
    \begin{equation*}
        \gamma|A|^2 - |Ae_1|^2 = \gamma \sum_{i,j} a_{ij}^2 - \sum_{i} a_{i1}^2 \ge \gamma \sum_i a_{ii}^2 - a_{11}^2 \ge \gamma a_{11}^2 +   \gamma \cdot \frac{1}{4}a_{11}^2 - a_{11}^2 \ge 0,
    \end{equation*}

    \noindent since $\frac{5}{4}\gamma \ge 1$. We have used that $\sum_{i=2}^5 a_{ii}^2 \ge \frac{1}{4} (\sum_{i=2}^5 a_{ii})^2 = \frac{1}{4}a_{11}^2$, which follows from the Cauchy-Schwarz inequality and the fact that $H=0$. 
    
    These computations were made at the point $p \in M$, which is arbitrary. Therefore, from \eqref{desigualdade que vem da hipotese de indice finito v1} we find
        \begin{equation*}
        \int_M \gamma|\nabla f|^2+\lambda_{Ric}^Mf^2 \ge c\gamma \int_M f^2
    \end{equation*}

        \noindent for every $f \in C^{\infty}_0(M \backslash B_R)$. This contradicts Theorem \ref{adaptacao do K. Xu dimensional constraints}.        
    \end{proof}

    \begin{obs}
        We now use Theorem \ref{quarto teorema principal k > 0} and Theorem \ref{corolario da extensao do teorema do Catino} to study finite index CMC hypersurfaces immersed in ambient spaces of the form $X:=\mathbb{S}^k \times \mathbb{R}^{6-k}$, where $ 0 \le k \le 6$, which are Riemannian product spaces between a round sphere and a Euclidean factor. Let $M$ be a complete two-sided finite index CMC hypersurface immersed in $X$. By Theorem \ref{quarto teorema principal k > 0}, $M$ is either minimal or compact. If $k = 5$ or $k=6$, then $X$ further satisfies the curvature hypothesis of Theorem \ref{corolario da extensao do teorema do Catino} and we conclude that $M$ is necessarily compact. On the other hand, if $t_0 \in \mathbb{R}$ and $k \le 4$, then $\mathbb{S}^{k} \times \mathbb{R}^{5-k} \times \{t_0\}$ is a complete non-compact two-sided strongly stable minimal hypersurface embedded in $\mathbb{S}^{k} \times \mathbb{R}^{6-k}$, and thus Theorem \ref{quarto teorema principal k > 0} is optimal in this class of examples.
    \end{obs}

\appendix

\section{Lower dimensions} \label{section::Lower dimensions}

In this section, we show that the methods developed in this paper provide new proofs of results due to X. Cheng \cite{XuCheng} and Q. Deng \cite{Qintao} on the compactness of complete finite index CMC hypersurfaces immersed in Riemannian manifolds $X^{n+1}$, where $n+1 \in \{4,5\}$. The reader is referred to the work of A. da Silveira \cite{AlexandreDaSilveira}, F. Lopez and A. Ros \cite{LopezRos}, D. Fischer-Colbrie and R. Schoen \cite{fischer-colbrie--schoen} and K. Frensel \cite{K.Frensel--long} for results of a similar nature when the ambient dimension is three. Recall the notation from \eqref{def de lambda-Ricci} and \ref{def de Lambda-bi-Ricci}; here we also use the notation $\lambda_{BRic}^X (p):= \Lambda_1^X(p)$.

		\begin{teo}[\textit{cf.} \cite{XuCheng}, Proposition 2.1] \label{teorema da Xu Cheng}
                Let $X^{n+1}$ be a Riemannian manifold of dimension $n+1 \in \{4,5\}$. Suppose that $\lambda_{BRic}^X \ge - \frac{(5-n)}{4}H^2 + \delta$ for some $\delta>0$ and $H \in \mathbb{R}$. Let $M$ be a complete, two-sided CMC hypersurface immersed in $X$ with mean curvature $H$. If $M$ has finite index, then $M$ is compact.
		\end{teo}

\begin{proof} 
    Suppose, by contradiction, that there exists a complete, non-compact, two-sided CMC hypersurface $M$ immersed in $X$ with finite index and mean curvature $H$.

    The finite index hypothesis guarantees that there exists $R>0$ such that
    \begin{equation}
        \int_M |\nabla f|^2 \ge \int_M (Ric(\nu)+ |A|^2)f^2   \label{desigualdade que vem da hipotese de indice finito}
    \end{equation}

    \noindent for every $f \in C^{\infty}_0(M \backslash B_R)$. Taking traces on the Gauss formula, we prove that
\begin{equation} \label{primeiro traço na eq Gauss}
    Ric_X(\nu) = BRic_X(\nu) - Ric_M(z,z)  + H \langle Az,z \rangle - |Az|^2,
\end{equation} 

\noindent for any unit vector $z \in TM$. Moreover, by Lemma 2.1 in \cite{XuCheng}, 
\begin{equation} \label{estimativa da Xu Cheng}
    H \langle Az,z \rangle - |Az|^2 +|A|^2 \ge \frac{(5-n)}{4}H^2,
\end{equation}

\noindent for any unit vector $z \in TM$. For a fixed $p \in M$, we pick $z \in T_p M$ so that $Ric_M(z,z) = \lambda^M_{Ric}(p)$ and use \eqref{primeiro traço na eq Gauss} and \eqref{estimativa da Xu Cheng} to prove $\lambda_{Ric}^M + Ric(\nu)+ |A|^2  \ge \lambda_{BRic}^X + \frac{(5-n)}{4}H^2$. 

Therefore, we obtain from \eqref{desigualdade que vem da hipotese de indice finito} and the curvature hypothesis on $X$ that 
    \begin{equation*}   
        \int_M |\nabla f|^2 + \lambda_{Ric}^M f^2  \ge \delta \int_M f^2 ,
    \end{equation*}

    \noindent for every $f \in C^{\infty}_0(M \backslash B_R)$. This contradicts Theorem \ref{adaptacao do K. Xu dimensional constraints}.      
\end{proof}

Notice that the hypotheses of the theorem regarding the curvature of the ambient space and the mean curvature of the hypersurface are verified when $X$ has non-negative bi-Ricci curvature and $M$ is non-minimal. The theorem also applies when $X$ has uniformly positive bi-Ricci curvature and $M$ is minimal. When $X$ is the hyperbolic space of dimension $4$ or $5$, the claim of the above result has been improved and we will give a new proof of this improved version shortly.

One of the main difficulties in the proof of Theorem \ref{teorema da Xu Cheng} lies in the estimate of the Ricci curvature of the CMC hypersurface in terms of its mean curvature and other extrinsic curvature terms. When the ambient space has constant curvature, one can improve this estimate, as we prove shortly. To our knowledge, the following two results, which can be deduced from the work of Q. Deng \cite{Qintao}, contain the best known estimates for the critical mean curvature value defined in Remark \ref{existence da curvatura media critica remark}, when the ambient space is one of the hyperbolic spaces $\mathbb{H}^4$ or $\mathbb{H}^5$.

\begin{teo}[\textit{cf.} \cite{Qintao}] \label{primeiro teo do Qintao}
    Let $M$ be a complete CMC hypersurface immersed in $\mathbb{H}^4$ with mean curvature $|H| > 3+\varepsilon$, for $\varepsilon = \frac{1}{7}(8\sqrt{7}-21)$. If $M$ has finite index, then $M$ is compact.  
\end{teo}

\begin{teo}[\textit{cf.} \cite{Qintao}] \label{segundo teo do Qintao}
    Let $M$ be a complete CMC hypersurface immersed in $\mathbb{H}^5$ with mean curvature $|H| > 4+\varepsilon$, for $\varepsilon = \frac{1}{37}(10\sqrt{259}-148)$. If $M$ has finite index, then $M$ is compact.  
\end{teo}

The approximate values of the constants $\varepsilon$ above are $\varepsilon = \frac{1}{7}(8\sqrt{7}-21) \sim 0.024$ and $\varepsilon = \frac{1}{37}(10\sqrt{259}-148) \sim 0.350$, respectively. We give a proof for Theorem \ref{primeiro teo do Qintao}, but the same strategy works to prove Theorem \ref{segundo teo do Qintao}.

\begin{proof}[Proof of Theorem \ref{primeiro teo do Qintao}]
    Suppose, to reach a contradiction, that there exists a complete and non-compact CMC hypersurface $M$ with mean curvature $H > 3+\varepsilon$ immersed in $\mathbb{H}^4$ with finite index. By the Gauss equation, we have $Ric_M =  -(n-1)g + HA - A^2$. Hence $Ric_M \circ A = A \circ Ric_M$, so at every $x \in M$, the smallest eigenvalue of $Ric_x^M$ is attained by an eigenvector of the second fundamental form $A$ of $M$. 

       Therefore, $\lambda_{Ric} = -2 + \mu_1 H - \mu_1^2$, where $\mu_1$ is some eigenvalue of $A$. Let $p \in M$ and $R>0$ be such that $M \backslash B_R(p)$ is strongly stable. Then for every $f \in C^\infty_0(M\backslash B_R(p))$ we have
    \begin{equation*}
         \int_M 2|\nabla f|^2 + \lambda_{Ric}f^2 \ge \int_M \Big( 2 (-3+|A|^2) + \lambda_{Ric} \Big)f^2.
    \end{equation*}

    At each point, we compute, introducing the traceless second fundamental form $\Phi$,
    \begin{equation*}
    \begin{split}
          2 (-3+|A|^2) + \lambda_{Ric}  = 2(-3 + \frac{H^2}{3} + |\Phi|^2) + (-2 + \mu_1 H - \mu_1^2).
    \end{split}
    \end{equation*}

    Now we denote by $\nu_i := \mu_i - \frac{H}{3}$ the eigenvalues of $\Phi$ and use $H=3+\varepsilon$ to compute   \begin{equation*}
    \begin{split}
          2 (-3+|A|^2) + \lambda_{Ric}  &= 2(-3 + \frac{9+6\varepsilon + \varepsilon^2}{3} + |\Phi|^2) + (-2 + (\nu_1 +\frac{H}{3})(3+\varepsilon) - (\nu_1+\frac{H}{3})^2) \\
          &= 2|\Phi|^2 + C(\varepsilon) + (1+\frac{1}{3}\varepsilon)\nu_1 - \nu_1^2 \\
          &= C(\varepsilon) + \nu_1^2 + (1+\frac{1}{3}\varepsilon)\nu_1 +2(\nu_2^2+\nu_3^2) \\
          &\ge  C(\varepsilon) + 2\nu_1^2 + (1+\frac{1}{3}\varepsilon)\nu_1
    \end{split}
    \end{equation*}

    \noindent where we have used the inequality $2(\nu_2^2 + \nu_3^2) \ge (\nu_2+\nu_3)^2 = \nu_1^2$ that follows form the fact that $\Phi$ is traceless, and we have used the Cauchy-Schwarz inequality. Here $C(\varepsilon)=\frac{1}{3}(16\varepsilon+\frac{8}{3}\varepsilon^2)$. We compute
    \begin{equation*}
    \begin{split}
        2\nu_1^2 + (1+\frac{1}{3}\varepsilon)\nu_1 &= 2\Big(\nu_1^2 + \frac{1}{2}(1+\frac{1}{3}\varepsilon)\nu_1 \Big) \ge - \frac{1}{8}(1+\frac{1}{3}\varepsilon)^2
    \end{split}
    \end{equation*}

    \noindent and note that 
    \begin{equation*}
        C(\varepsilon) >  \frac{1}{8}(1+\frac{1}{3}\varepsilon)^2 \Leftrightarrow (48\varepsilon+8\varepsilon^2) >  \frac{1}{8}(9+6\varepsilon+\varepsilon^2) \Leftrightarrow \varepsilon> \frac{1}{7}(8\sqrt{7}-21)
    \end{equation*}

    Since this is our hypothesis, we arrive at a contradiction due to Theorem \ref{adaptacao do K. Xu dimensional constraints}.
\end{proof}

Finally, we show that the strategy developed in the proof of Theorem \ref{CMC de indice finito em ambiente com curvatura limitada por baixo} allows for an alternative proof of the following result of H. Hong \cite{HongH4}.

            \begin{teo}[\textit{cf.} \cite{HongH4}] \label{teo do Hong em H4}
                Let $M^3$ be a complete finite index CMC hypersurface immersed in the hyperbolic space $\mathbb{H}^4$ with mean curvature $|H|>3$. If $M$ has a finite number of ends and $H^1_c(M; \mathbb{R})$ is finite dimensional, then $M$ is compact.
            \end{teo}
            
        \begin{proof}

            In order to reach a contradiction, suppose that there exists a non-compact $M^3$ satisfying the hypotheses of the theorem. We know that $M$ must have infinite volume (\textit{cf.} Theorem \ref{Katia Frensel volume infinito}). By Theorem \ref{estimativa de curvatura mais geral possivel na intro}, $M$ has bounded second fundamental form. 
            
            Now we use $\mu$-bubble techniques to further restrict the geometry of $M$. It follows from the Schoen-Yau rearrangement of the stability inequality that 
            \begin{equation*}
                \lambda_1(-\Delta_M + \frac{1}{2}R_M) \ge \delta > 0
            \end{equation*}

            \noindent for some $\delta>0$, because $|H|>3$. Here $R_M$ is the scalar curvature of $M$. 

             Using Lemma 6.1 in \cite{Chodosh-Li-R4-mu-bubbles} we construct an exhaustion of $M$, $\Omega_n$, by precompact smooth open sets such that every connected component $\Sigma$ of $\partial \Omega_n$ has controlled volume: $|\Sigma| \le V = V(\delta)$. We can assume that each connected component of the complement of $\Omega_n$ is unbounded. The topological assumptions on $M$ allow us to guarantee that the number of connected components of $\Omega_n$ eventually stabilizes to a constant number (\textit{cf.} Lemma 6.1 in \cite{Chodosh-Li-PrimeiraProva-R4}).
            
            This geometric property shows that the Cheeger constant of $M$ is zero, because it has infinite volume (see Theorem \ref{Katia Frensel volume infinito}). But this contradicts P. Buser's inequality (\cite{Buser}, Theorem 7.1), because by Theorem \ref{CMC com H>Hlambda tem lambda1>0} we know that $\lambda_1(M)>0$.

        \end{proof}

	\bibliographystyle{plain}
	\bibliography{refs}
\end{document}